%% file: convolution.tex
\documentclass[smallcondensed]{svjour3}

\makeatletter
\let\cl@chapter\undefined
\makeatletter

\usepackage{amsmath}
\usepackage{amsfonts}
\usepackage{amssymb}
\usepackage{hyperref}
\usepackage{cleveref}
\usepackage[utf8]{inputenc}
\usepackage{algorithm}
\usepackage{algpseudocode}
\usepackage{comment}
\usepackage{color}
\usepackage{tikz,pgfplots,pgfplotstable}
\usetikzlibrary{decorations.pathreplacing,arrows,decorations,backgrounds,positioning}
\usepackage{pgfplots}

\usepackage{subcaption}
\usepackage{graphicx}
\usepackage{pgfplots}
\usepackage{pgfplotstable}
\pgfplotsset{width=7cm,compat=1.8}

\usepackage{multirow}

\usepackage{listings}
\usetikzlibrary{external}
\newcommand{\ttg}{\texttt{g}}
\newcommand{\ttu}{\texttt{u}}

\newcommand{\tty}{\mathtt{y}}
\newcommand{\ttf}{\mathtt{f}}
\newcommand{\ttK}{\mathtt{K}}

\newcommand{\K}{\hat{k}}

\newcommand{\eps}{\varepsilon}
\def\Th{\mathcal{T\!}_h}
\def\Ths{\mathcal{T\!}_{h'}}
\def\ATh{\mathcal{AT\!}_{h}}

\newcommand{\bfn}{\mathbf{n}}

\newcommand{\CC}{\mathbb{C}}
\newcommand{\NN}{\mathbb{N}}
\newcommand{\RR}{\mathbb{R}}

\newcommand{\cH}{\mathcal{H}}
\newcommand{\cL}{\mathcal{L}}
\newcommand{\cO}{\mathcal{O}}
\newcommand{\cP}{\mathcal{P}}

\renewcommand{\d}{\operatorname{d}\!}

\DeclareMathOperator*{\argmin}{arg\,min}
\DeclareMathOperator\acosh{acosh}

\newcommand{\bin}{\operatorname{bin}}

\newcommand{\isdef}{\mathrel{\mathrel{\mathop:}=}}
\newcommand{\defis}{\mathrel{=\mathrel{\mathop:}}}
\DeclareMathOperator*{\argmax}{arg\,max}

\title{A fast and oblivious matrix compression algorithm for Volterra integral operators}

\author{J.~D\"olz \and H.~Egger \and V.~Shashkov}

\institute{
J.~D\"olz \at
Institute for Numerical Simulation,
University of Bonn,
Friedrich-Hirzebruch-Allee 7,
53115 Bonn,
Germany.
\email{doelz@ins.uni-bonn.de}
\and
H.~Egger \at
Numerical Analysis and Scientific Computing,
Department of Mathematics, 
TU Darmstadt, Dolivostr.~15,
64293 Darmstadt,
Germany.
\email{egger@mathematik.tu-darmstadt.de}
\and
V.~Shashkov \at
Numerical Analysis and Scientific Computing,
Department of Mathematics, 
TU Darmstadt, Dolivostr.~15,
64293 Darmstadt,
Germany.
\email{shashkov@mathematik.tu-darmstadt.de}
}

\begin{document}

\maketitle

\begin{abstract}
    The numerical solution of dynamical systems with memory requires the efficient evaluation of Volterra integral operators in an evolutionary manner. After appropriate discretisation, the basic problem can be represented as a matrix-vector product with a lower diagonal but densely populated matrix. For typical applications, like fractional diffusion or large scale dynamical systems with delay, the memory cost for storing the matrix approximations and complete history of the data then becomes prohibitive for an accurate numerical approximation.
    For Volterra-integral operators of convolution type, the \emph{fast and oblivious convolution quadrature} method of Sch\"adle, Lopez-Fernandez, and Lubich resolves this issue and allows to compute the discretized evaluation with $N$ time steps in $O(N \log N)$ complexity and only requires $O(\log N)$ active memory to store a compressed version of the complete history of the data. 
    We will show that this algorithm can be interpreted as an $\cH$-matrix approximation of the underlying integral operator. A further improvement can thus be achieved, in principle, by resorting to $\cH^2$-matrix compression techniques.  
    Following this idea, we formulate a variant of the $\cH^2$-matrix vector product for discretized Volterra integral operators that can be performed in an evolutionary and oblivious manner and requires only $O(N)$ operations and $O(\log N)$ active memory. In addition to the acceleration, more general asymptotically smooth kernels can be treated and the algorithm does not require a-priori knowledge of the number of time steps.
    The efficiency of the proposed method is demonstrated by application to some typical test problems.
    \keywords{Volterra integral operators \and convolution quadrature \and $\cH^2$-matrices \and matrix compression}
\end{abstract}

\section{Introduction}

We study the numerical solution of dynamical systems with memory which can be modelled by abstract Volterra integro-differential equations of the form
\begin{align}\label{eq:integrodiff}
\alpha(t) y'(t) + A(t)y(t) = \int_0^t k(t,s) f(s,y(s)) \d s, \qquad 0 \le t \le T.
\end{align}
Such problems arise in a variety of applications, e.g., in anomalous diffusion \cite{MK2000}, neural sciences \cite{Ama1977}, transparent boundary conditions \cite{AGH2000,Hag1999,JG2004,JG2008}, wave propagation \cite{AGH2000,DES2020,Hag1999}, field circuit coupling \cite{ESS2020} and many more, see also \cite{Bru2004,Bru2017a,Lub2004,Say2016} and the references therein.
The simplest model problem which already shares the essential difficulties stemming from non-locality of the right hand side in \eqref{eq:integrodiff} is the evaluation of the integral operator
\begin{align} \label{eq:integral}
    y(t) = \int_0^t k(t,s) f(s) \, \d s, \qquad 0 \le t \le T,
\end{align}
with kernel function $k$, data $f$, and result function $y$.
Let us emphasize that, in order to allow the application in the context of integro-differential problems \eqref{eq:integrodiff}, the parameter-dependent integrals \eqref{eq:integral} have to be evaluated in an \emph{evolutionary} manner, i.e., for successively increasing time. 
The results obtained for \eqref{eq:integral} then quite naturally extend to \eqref{eq:integrodiff}. We hence focus on the evaluation of Volterra integral operators \eqref{eq:integral} in the following and return to more general problems in \Cref{sec:numerics}.

\subsection{Discretisation and related work}
After applying some appropriate discretisation procedure, see \cite{Bru2004} for a survey, problem \eqref{eq:integral} can be phrased as a simple matrix-vector multiplication 
\begin{align} \label{eq:algebraic}
\tty_n = (\ttK \ttf)_n, \qquad 1 \le n \le N.
\end{align}
The evolutionary character and the nonlocal interactions are reflected by the fact that the matrix $\ttK \in \RR^{N \times N}$ is lower block triangular but densely populated. 
The straight-forward computation of the result vector $\tty$ requires $O(N^2)$ algebraic operations. 
The evolutionary character of problem \eqref{eq:integral} can be preserved by computing the entries $\tty_n$ for $n=1,\ldots,N$ recursively, i.e., by traversing the matrix $\ttK$ from top to bottom.
If, on the other hand, the matrix $\ttK$ is traversed from left to right, then the algorithm becomes \emph{oblivious}, i.e., the data $\ttf_n$ is only required in the $n$th step of the algorithm, 
but the execution of \eqref{eq:algebraic} then requires $O(N)$ active memory to store the partial sums for every row. 
Although the evaluation can then still be organized in an evolutionary manner, see \Cref{sec:uniform}, the number of time steps $N$ needs to be fixed \emph{a-priori} in order to store the intermediate results.

For the particular case that the integral kernel in \eqref{eq:integral} is of convolution type
\begin{align} \label{eq:convolution_kernel} 
k(t,s) = k(t-s),
\end{align}
a careful discretisation of \eqref{eq:volterra} gives rise to an algebraic system \eqref{eq:algebraic} with block Toeplitz matrix $\ttK$, and the discrete solution $\tty$ can be computed in $O(N \log N)$ operations using fast Fourier transforms. 
As shown in \cite{HLS1985}, an evolutionary version of the matrix vector product can be realized in $O(N \log^2 N)$ complexity and requiring $O(N)$ active memory.
The convolution quadrature methods of \cite{Lub1988,Lub1988a,LO1993} treat the case that only the Laplace transform $\K(s)$ of the convolution kernel \eqref{eq:convolution_kernel} is available.
The \emph{fast and oblivious convolution quadrature} method introduced in \cite{LPS2006,LS2002} allows the efficient evaluation of Volterra integrals with convolution kernel in an \emph{evolutionary} and \emph{oblivious} manner with $O(N \log N)$ operations and only $O(\log N)$ active memory and $O(\log N)$ evaluations of the Laplace transform $\widehat k(s)$. This method is close to optimal concerning complexity and memory requirements and has been applied successfully to the numerical solution of partial differential equations with transparent boundary conditions \cite{Hag1999}, the efficient realization of boundary element methods for the wave equation \cite{Say2016} or fractional diffusion \cite{CLP2006}.

For integral operators \eqref{eq:integral} with general kernels $k(t,s)$, the above mentioned methods cannot be applied directly.
Alternative approaches, like the fast multipole method \cite{FD2009,GR1987,Rok1985}, the panel clustering technique \cite{HN1989},  $\cH$- and $\cH^2$-matrices \cite{Bor2010,Hac1999}, multilevel techniques \cite{BL1990,Gie2001} or wavelet algorithms \cite{DPS1993},  which were developed and applied successfully in the context of molecular dynamics and boundary integral equations, are however still applicable.
These methods are based on certain hierarchical approximations for the kernel functions $k(t,s)$, whose error can be controlled
under appropriate smoothness assumptions, e.g. if the kernel is \emph{asymptotically smooth}; see \eqref{eq:smooth} for details.
If the data $f$ is independent of the solution $y$, the numerical evaluation of the Volterra-integral operator \eqref{eq:integral} can then be realized with $O(N \log^\alpha \!N)$ computational cost with some $\alpha \ge 0$ and $N$ again denoting the dimension of the underlying discretisation.
Moreover, data-sparse approximations of the matrix $\ttK$ for asymptotically smooth kernels $k(t,s)$ can be stored efficiently with only $O(N \log^\alpha \!N)$ memory and for convolution kernels $k(t-s)$ even with $O(\log N)$ memory; we refer to \cite{Bor2010,Hac2015} for details and an extensive list of references.
Unfortunately, the algorithms mentioned in literature are not evolutionary and, therefore, cannot be applied to more complex problems like \eqref{eq:integrodiff} directly.

\subsection{A fast and oblivious evolutionary algorithm}

In this paper, we propose an algorithm for the efficient evaluation of Volterra integrals \eqref{eq:integral} 
or corresponding matrix-vector products \eqref{eq:algebraic} which shares the benefits and overcomes the drawbacks of the approaches mentioned above, i.e., it is
\begin{itemize}
\item \emph{evolutionary}: the approximations $\tty_n$ can be computed one after another and the number of time steps $N$ does not need to be known in advance,
\item \emph{oblivious}: the entry $\ttf_n$ of the right hand side is only required in the $n$th step,
\item \emph{fast}: the evaluation of all $\tty_n$, $1 \le n \le N$ requires only $\cO(N)$ operations, and
\item \emph{memory efficient}: the storage of the convolution matrix requires only $\cO(N)$ memory for general and $\cO(\log N)$ memory for convolution kernels. The matrix entities can also be computed on the fly, such that only $\cO(\log N )$ storage is required to store a compressed history of the data $f$.
\end{itemize}
Our strategy is based on the ideas of polynomial-based $\cH^2$-compression algorithms for finding hierarchical low-rank approximations of the kernel function $k(t,s)$ leading to a block-structured hierarchical approximation of the matrix $\ttK$ in \eqref{eq:algebraic}.
The accuracy of the underlying approximation can thus be guaranteed by well-known approximation results; see \cite{Bor2010,Hac2015} for instance.
A key ingredient for our considerations is the one-dimensional nature of the integration domain which allows to characterize the block structure of the approximating hierarchical matrix explicitly.
This allows us to formulate an algorithm which traverses the compressed matrix $\ttK$ from top-to-bottom in accordance with the evolutionary structure of the underlying problem.
The hierarchical approximation of the convolution kernel also yields a compression strategy for the history of the data $f$. 
In this sense, our algorithm can be considered a generalisation of \cite{AGH2000,BH2017a,JG2004,JG2008}, where a fast multipole expansion was employed to accelerate the \emph{sum of exponentials approach}, or to \cite{KLR1996}, where a polynomial on growing time steps was employed for the compression of the data, as well as to \cite{KG2021}, where an evolutionary 
$\cH$-matrix approximation with a special low-rank structure was constructed.
As a further result, we show that our algorithm seamlessly integrates into the convolution quadrature framework of \cite{Lub1988,Lub1988a}, when the kernel $k(t-s)$ is of convolution type and only accessible via its Laplace transform. 
In analogy to the treatment of nearfield contributions in the fast boundary element method, we utilize standard convolution quadrature to compute the entries of the convolution matrix close to the diagonal, 
while numerical inverse Laplace transforms \cite{DW2015} are used to setup an $\cH^2$-approximation of the remaining off-diagonal parts of the convolution matrix in the time domain. 
This approach has some strong similarities to the fast and oblivious convolution quadrature method \cite{LS2002,SLL2006}, but we will reveal some important differences. In particular, we illustrate that the methods of \cite{LS2002,SLL2006} can actually be interpreted as $\cH$-matrix approximations with a particular organisation of the matrix-vector product in \eqref{eq:algebraic}, which shows that the $\cO(N\log N)$ complexity cannot be further improved. Moreover, the convolution matrix must be applied from left to right to allow for an oblivious evaluation and the number of time steps $N$ must be known in advance.
In contrast to that, our new algorithm is based on an $\cH^2$-approximation of the matrix $\ttK$ and the evolutionary, fast and oblivious evaluation of the matrix-vector product can be realized by traversing through the matrix from top to bottom in $\cO(N)$ complexity and without needing to know the number of time steps $N$ in advance.
Finally, our algorithm naturally extends to general integral kernels $k(t,s)$ increasing the field of applications substantially.

\subsection{Outline}
In \Cref{sec:time} we recall some general approximation results, introduce our basic notation, and state a slightly modified algorithm for the dense evaluation of the Volterra integral operators to illustrate some basic principles that we exploit later on.
\Cref{sec:fo} is concerned with a geometric partitioning on the domain of integration, the multilevel hierarchy used for the $\cH^2$-compression, and the description and analysis of our new algorithm.
In \Cref{sec:frequency} we consider convolution kernels $\K(s)$ and discuss the relation of our algorithm to Lubich's convolution quadrature and the connections to the fast and oblivious algorithm of \cite{LS2002,SLL2006}. 
To support our theoretical considerations, some numerical results are provided in \Cref{sec:numerics}.

\vfill 

\pagebreak

\section{Preliminary results} \label{sec:time}

Let us start with summarizing some basic results about typical discretisation strategies for Volterra integral operators 
\begin{align} \label{eq:volterra}
    y(t) = \int_0^t k(t,s) f(s) ds
\end{align}
which are the basis for the efficient and reliable numerical evaluation later on.
For simplicity, all functions $y$, $f$, and $k$ are assumed to be scalar valued.
We will demonstrate the application to more general problems of the form \eqref{eq:integrodiff} in \Cref{sec:numerics}.

\subsection{A general approximation result} \label{sec:general}

For the discretisation of the integral operator \eqref{eq:volterra}, we consider methods of the form 
\begin{align} \label{eq:abstract}
    \widetilde y_h(t) = \int_0^t  k_h(t,s) f_h(s) ds,
\end{align}
where $k_h$ and $f_h$ are suitable approximations for $k$ and $f$. 
The subscript $h$ will be used to designate approximations throughout.
The following result may serve as a theoretical justification for a wide variety of particular discretisation schemes.

\begin{lemma} \label{lem:abstract}
Let $T>0$, kernels $k,k_h \in L^\infty(0,T;L^r(0,T))$, and $f,f_h \in L^{r'}(0,T)$ be given with $1 \le r,r' \le \infty$ with $1/r+1/r'=1$.
Further assume that 
\begin{align} \label{eq:dataerror}
\|k - k_h\|_{L^\infty(0,T;L^r(0,T))} \le \eps
\qquad \text{and} \qquad 
\|f - f_h\|_{L^{r'}(0,T)} \le \eps. 
\end{align}
Then the functions $y$, $\widetilde y_h$ defined by \eqref{eq:volterra} and \eqref{eq:abstract} satisfy
\begin{align} \label{eq:approximation}
\|y - \widetilde y_h\|_{L^\infty(0,T)} \le C (\|k\|_{L^r(0,T)} + \|f\|_{L^{r'}(0,T)} + \eps) \, \eps,
\end{align}
i.e., the error in the results can be bounded uniformly by the perturbation in the data.
\end{lemma}
\begin{proof}
From H\"older's inequality, we can deduce that 
\begin{align*}
|y(t) &- \widetilde y_h(t)| 
 \le \int_0^t |k(t,s)| |f(s) - f_h(s)|  + |k(t,s) - k_h(t,s)| |f_h(s)| ds \\
&\le \|k(t,\cdot)\|_{L^r(0,T)} \|f - f_h\|_{L^{r'}(0,T)} + \|k(t,\cdot) - k_h(t,\cdot)\|_{L^r(0,T)} \|f_h\|_{L^{r'}(0,T)}.
\end{align*}
The result then follows by estimating $\|f_h\| \le \|f\| + \|f - f_h\|$, using the estimates for the differences in the data, and taking the supremum over all $0 < t < T$.
\hfill \qed
\end{proof}

\begin{remark}
The constant $C$ in the estimate \eqref{eq:approximation} depends on the kernel $k$, but is independent of $T$. The result can therefore be applied to time intervals of arbitrary size. 
Without substantially changing the argument, it is possible to obtain similar estimates also in other norms.
In many cases, $\widetilde y_h$ only serves as an intermediate result and the final approximation is given by
$y_h(t) = (P_h \widetilde y_h)(t)$,
where $P_h$ is some projection or interpolation operator; a particular case will be discussed in more detail below.
Estimates for the error $\|y - y_h\|$ can then be obtained by additionally taking into account the projection errors.
\end{remark}

\subsection{Piecewise polynomial approximations} \label{sec:uniform}

Many discretisation methods for integral or integro-differential equations, e.g., collocation or Galerkin methods \cite{Bru2004}, are based on piecewise polynomial approximations and fit into the abstract form mentioned above. 
As a particular example and for later reference, we consider such approximations in a bit more detail now.

Let $h>0$ be given, define $t^n = n h$, $n \ge 0$, and set $T=t^N=Nh$.
We set $I^n = [t^{n-1},t^n]$ for $1 \le n \le N$, and denote by $\Th=\{ I^n : 1 \le n \le N\}$ the resulting uniform mesh of the interval $[0,T]$. 
We write $\cP_p(a,b)$ for the space of polynomials of degree at most $p$ over the interval $(a,b)$, and $\cP_{q,q}((a,b) \times (c,d))= \cP_q(a,b) \otimes \cP_q(c,d)$ for the space of polynomials in two variables of degree at most $q$ in each variable.
We further define piecewise polynomial spaces
\begin{align}
    \cP_p(\Th) &= \{f \in L^1(0,T) : f|_{I^n} \in \cP_p(I^n)\}, \label{eq:polyspacedata}\\
    \cP_{q,q}(\Th \times \Th) &= \{k \in L^1((0,T) \times (0,T)) : k|_{I^m \times I^n}\in \cP_{q,q}(I^m \times I^n)\},
\end{align}
over the grid $\Th$ and the tensor-product grid $\Th \times \Th$.

For sufficiently regular functions $f$ and $k$ over the mesh $\Th$ and $\Th \times \Th$, piecewise polynomial approximations $k_h$, $f_h$ satisfying \eqref{eq:dataerror} can be found by appropriate interpolation and choosing the mesh size $h$ small enough. 
Without further structural assumptions on the data, it seems natural to use uniform grids $\Th$, which can be obtained, e.g., by  uniform refinement of some reference grid. 
In \Cref{fig:uniform}, we depict the resulting uniform partitions for approximation of the kernel function $k$.
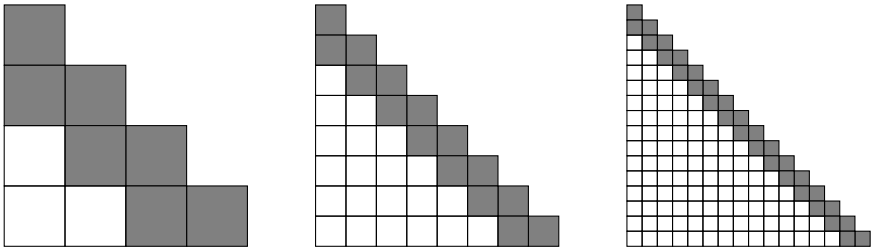
\begin{figure}[ht!]
\centering
\hspace*{2em}
\input{figures/uniform-mesh-L1}
\hspace*{2em}
\input{figures/uniform-mesh-L2}
\hspace*{2em}
\input{figures/uniform-mesh-L3}
\caption{Uniformly refined grids $\Th \times \Th$ for approximation of $k$. Only the elements required for approximating $k(t,s)$ for $s \le t$ are depicted. The grid cells near the diagonal $t=s$, i.e, the \emph{nearfield}, play a special role and are thus coloured in gray.\label{fig:uniform}}
\end{figure}

\noindent
The evaluation of $\widetilde y_h$ defined by \eqref{eq:abstract} 
can be split into two contributions 
\begin{align} \label{eq:split}
\widetilde y_h(t) 
=\widetilde w_h(t) + \widetilde z_h(t)
\end{align}
with $w_h$ corresponding to the integrals over the \emph{farfield} cells and $z_h$ over the \emph{nearfield} cells, which are depicted in white and gray in \Cref{fig:uniform}, respectively. 
As discussed in the following subsection, the numerical treatment of these contributions differs slightly.

\subsection{Practical realisation}\label{sec:uniform2}

From equation \eqref{eq:abstract} and the choice of $f_h \in \cP_p(\Th)$ and $k_h \in \cP_{q,q}(\Th \times \Th)$, one can see that $\widetilde y_h$ is a piecewise polynomial of degree $\le p+q+1$ over the grid $\Th$. 
It is often convenient to replace $\tilde y_h$ by a piecewise polynomial $y_h \in \cP_p(\Th)$ with the same degree as the data.
For this purpose, we simply choose a set of distinct points $0 \le \gamma_j \le 1$, $0 \le j \le p$, and define $y_h \in \cP_p(\Th)$ by collocation
\begin{align} \label{eq:collocation}
y_h(t^n_j) = \widetilde y_h(t^n_j), \qquad 0 \le j \le p, 
\end{align}
on every time interval $I^n \in \Th$,
with collocation points $t^n_j=t^{n-1}+\gamma_j h$, $j=0,\ldots, p$. 
Now let $\psi_j^n$, $j=0,\ldots,p$, denote the Lagrangian basis of $\cP_p(I^n)$ with respect to the interpolation points $t^n_j$, i.e., 
\begin{align} \label{eq:psi_lagrangian}
    \psi_i^n(t_j^n) = \delta_{i,j}, 
    \qquad 0 \le i \le p, \quad 1 \le n \le N.
\end{align}
Then as a consequence of the uniformity of the mesh $\Th$, one may deduce that 
\begin{align} \label{eq:psi_invariant} \psi_i^n(t-t^n)=\psi_i^m(t-t^m), \qquad 0 \le i \le p, \quad 1 \le m,n \le N,
\end{align}
i.e., the basis $\{\psi_i^n\}_{0 \le i \le p}$ is \emph{invariant under translation}, which will become an important ingredient for our algorithm below.
The approximate data $f_h$ and the discrete solution $y_h$ can now be expanded as
\begin{align} \label{eq:yhfh}
y_h(t)=\sum_{j=0}^p y_j^n\psi_j^n(t),
\qquad
f_h(t)=\sum_{j=0}^p f_j^n\psi_j^n(t),
\qquad
\text{for } t\in I^n.
\end{align}
In a similar manner, the approximate kernel function $k_h$ can be expanded with respect to a set of bases $\{\varphi_i^n\}_{i=0,\ldots,q}$ for the spaces $\cP_q(I^n)$, which leads to 
\begin{align} \label{eq:kh}
k_h(s,t)=\sum_{i=0}^q\sum_{j=0}^q k_{i,j}^{m,n}\varphi_i^m(s)\varphi_j^n(t),
\qquad
\text{for } s\in I^m, \ t\in I^n.
\end{align}
We will again assume translation invariance of this second basis, i.e., 
\begin{align} \label{eq:phi_invariant}
\varphi_i^n(t-t^n)=\varphi_i^m(t-t^m),
\qquad 0 \le i \le q, \quad 1 \le m,n \le N.
\end{align}
Note that we allow for different polynomial degrees $q \ne p$ in the approximations $y_h, f_h$, and $k_h$, and hence two different sets of basis functions are required.
Evaluating \eqref{eq:abstract} at time $t=t_j^m$ and utilizing \eqref{eq:collocation}, we obtain 
\begin{align}\label{eq:y_j^m_dense}
y_h(t_j^m)
=
\sum_{n=1}^{m-2}\int_{I^n}k_h(t_j^m,s)f_h(s)\d s + \int_{t_{j}^{m-2}}^{t_j^m}k_h(t_j^m,s)f_h(s)\d s,
\end{align}
where we used the splitting of the integration domain $(0,t_j^m)$ into subintervals of the mesh $\Th$ and an additional separation of farfield and nearfield contributions; see \Cref{fig:uniform}. 
By inserting the basis representations \eqref{eq:yhfh} and \eqref{eq:kh} for $y_h$, $f_h$, and $k_h$, the integrals in the farfield contribution can be expressed as 
\begin{align}
\int_{I^n}k_h(t_j^m,s)f_h(s)\d s
=
\sum_{i=0}^q\varphi_i^m(t_j^m)\sum_{k=0}^q k_{i,k}^{m,n}\sum_{r=0}^p \int_{I^n}\varphi_k^n(s)\psi_r^n(s)\d s\, f_r^n.
\end{align}
For convenience of presentation, let us introduce the short-hand notations 
\begin{align}\label{eq:PijandQij}
P_{i,j}=\varphi_i^m(t_j^m),
\qquad
Q_{k,r}=\int_{I^n}\varphi_k^n(s)\psi_r^n(s)\d s,
\end{align}
and note that the corresponding matrices $P$ and $Q$ are independent of the time steps $m$, $n$, due to the translation invariance conditions \eqref{eq:psi_invariant} and \eqref{eq:phi_invariant}.
The result vector $y^m$ containing entries $y^m_j=y_h(t_j^m)$ from \eqref{eq:y_j^m_dense} can then be expressed as  
\begin{align}
y^m = w^m + z^m
\end{align}
with farfield contribution $w^m$ given by
\begin{align}\label{eq:PandQ}
    w^m = P u^m, \qquad u^m = \sum_{n=1}^{m-2} k^{m,n} g^n, \qquad g^n = Q f^n,
\end{align}
where $k^{m,n}$ is the matrix containing the entries $k^{m,n}_{i,j}$. 
Further introducing the symbols $K^{m,n} = P k^{m,n} Q$, this may be stated equivalently as $w^m = \sum_{n=0}^{m-2} K^{m,n} f^n$. 
In a similar manner, we may represent the nearfield contribution $z^m$ by 
\begin{align}
    z^m = K^{m,m-1} f^{m-1} + K^{m,m} f^m,
\end{align}
with appropriate nearfield matrices $K^{m,m-1}$, $K^{m,m} \in \RR^{(p+1) \times (p+1)}$.

We denote by $\tty$, $\ttf \in \RR^{N (p+1)}$ the global vectors that are obtained by stacking the element contributions $y^m$, $f^n$ together. 
The computation of $\tty$ can then be written as matrix-vector product $\tty = \ttK  \ttf$ with block-matrix $\ttK \in \RR^{N (p+1) \times N(p+1)}$ consisting of blocks $K^{m,n}$ as defined above. 
A possible block-based implementation of this matrix-vector product can be realized as follows.

\begin{algorithm}
\caption{\label{alg:denseconvolution}Evaluation of Volterra integral operators for uniform meshes.}
\begin{algorithmic}
\For{$m=1,\ldots,N$}
\State $u=0$
\For{$n=1,\ldots,m-2$}
\State $u=u+k^{m,n}g^n$
\EndFor
\State $g^m=Q f^m$
\State $w^m=Pu$
\State $z^m=K^{m,m-1}f^{m-1}+K^{m,m}f^m$
\State $y^m=w^m+z^m$
\EndFor
\end{algorithmic}
\end{algorithm}
\begin{remark}
At first sight, this algorithm may seem more complicated than actually required. In fact, after generating the matrix blocks $K^{m,n} = P k^{m,n} Q$, the $m$th component of the result vector could also be computed as $y^m=\sum_{n=1}^m K^{m,n} f^n$. 
The above version, however, is closer to the algorithm developed in the next section. 
Moreover, it is \emph{evolutionary}, i.e., the entries of the vector $\tty$ are computed one after another, and \emph{oblivious} in the sense that only the blocks $f^{m-1}$ and $f^m$ are needed for the computation of $y^m$. 
Note, however, that all auxiliary values $g^n$, $n=1,\ldots,m-2$ are required to compute the block $y^m$ and therefore have to be kept in memory.
This will be substantially improved in \Cref{sec:adaptive}.
\end{remark}

\subsection{Computational complexity}
As indicated above, the computation of $y_h$ according to \eqref{eq:y_j^m_dense} can be phrased in algebraic form as a matrix-vector product
\begin{align} \label{eq:algebraic2}
\tty = \ttK \ttf, 
\end{align}
with $\tty$ and $\ttf$ denoting the coefficient vectors for $y_h$ and $f_h$, and a lower block triangular matrix $\ttK \in \RR^{N (p+1) \times N (p+1)}$. Note that the pattern of the matrix $\ttK$ is structurally the same as that of the tensor-product grid underlying the approximation of the kernel function $k$; see \Cref{fig:uniform}, with each cell corresponding to a block  
of size $(p+1) \times (p+1)$.
Thus, the the computation of $\tty=\ttK\ttf$ will in general require $O(p^2 N^2)$ operations and $O(p^2 N^2)$ memory to store the matrix $\ttK$. In addition, we require $O(p N)$ active memory to store two values of $f^n$ and the history of $g$.

\section{A fast and oblivious algorithm}\label{sec:fo}
The aim of this section is to introduce a novel algorithm which allows for a simultaneous compression of the matrix $\ttK$ used for the evaluation of \eqref{eq:algebraic2} and the history of the data stored in the vectors $g^n$, $n \ge 1$.
The underlying approximation is that of $\cH^2$-matrix compression techniques \cite{Bor2010,Hac2015}. 
We first collect some results about these hierarchical approximations and then state and analyze our algorithm.

\subsection{Multilevel partitioning}
For ease of presentation we will assume that the number of time steps is given as $N=2^L$ with $L\in\NN$ and define $h=T/N$.
Now let $I^{(n;1)}=I^n$ and introduce a hierarchy of nested partitions into subintervals
\begin{align*}
    I^{(n;\ell)}&=I^{(2n-1;\ell-1)}\cup I^{(2n;\ell-1)}
    =\Big[t^{2^{\ell-1}(n-1)},t^{2^{\ell-1}n}\Big],\qquad \ell > 1,
\end{align*}
of length $2^{\ell-1} h$ by recursive coarsening of the intervals; see \Cref{fig:hierarchy} for a sketch. 
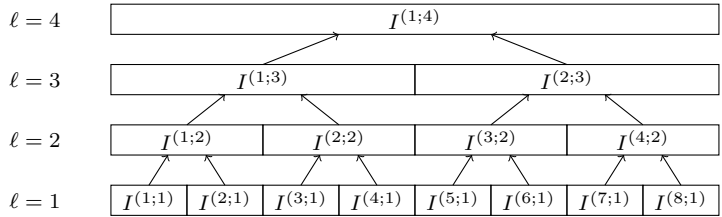
\begin{figure}[ht!]
\centering
\hspace*{2em}
\input{figures/mesh-1d}
\hspace*{2em}
\caption{Mesh hierarchy obtained by recursive coarsening of intervals $I^n=I^{(n;1)}$ with maximal corsening level $L=3$ and $N=2^L=8$ fine grid cells.
\label{fig:hierarchy}}
\end{figure}

\noindent
From the hierarchic construction, one can immediately see that 
\begin{align} \label{eq:Cnl}
I^n \subset I^{(C(n;\ell);\ell)} \qquad \text{for} \quad C(n;\ell):=\lceil n/2^{\ell-1}\rceil,
\end{align}
where $\lceil r\rceil$ denotes the smallest integer larger or equal to $r$ as usual.  

In a similar spirit to \cite{Bor2010,FD2009,GR1987,Hac2015}, we next introduce a multilevel partitioning of the support of the kernel $k$ leading to adaptive hierarchical meshes
\begin{align} \label{eq:ATh}
\begin{aligned}
\ATh=\{I^{(m;\ell)} \times &{}I^{(n;\ell)} \colon
\ell=1~\text{with}~n\in\{m-1,m\}~\text{or}\\
&{}I^{(m;\ell)}\cap I^{(n;\ell)}=\emptyset~\text{with}~I^{(\lceil m/2\rceil;\ell+1)}\cap I^{(\lceil n/2\rceil;\ell+1)}\neq\emptyset\}.
\end{aligned}
\end{align}
Note that every element of $\ATh$ is square and thus corresponds to one element of a, possibly coarser, uniform mesh $\Ths \times \Ths$.
Moreover, any element in $\ATh$ is the union of elements of the underlying uniform mesh $\Th \times \Th$ and can be constructed by recursive  agglomeration or coarsening.
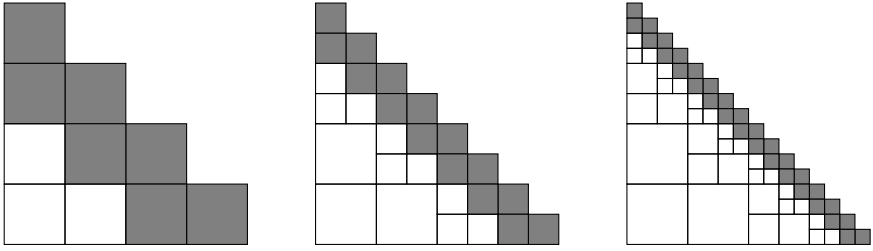
\begin{figure}[ht!]
\centering
\hspace*{2em}
\input{figures/adaptive-mesh-L1}
\hspace*{2em}
\input{figures/adaptive-mesh-L2}
\hspace*{2em}
\input{figures/adaptive-mesh-L3}
\caption{Adaptive hierarchical meshes $\ATh$ obtained by recursive coarsening of far fields cells in the corresponding uniformly refined meshes $\Th \times \Th$ in \Cref{fig:uniform}. 
\label{fig:adaptive}}
\end{figure}
As illustrated in \Cref{fig:adaptive}, the hierarchic adaptive mesh $\ATh$ can again be split into nearfield elements adjacent to the diagonal and the remaining far field elements.
Let us remark that the resulting partitioning and its splitting coincide with most of the classical partitioning strategies developed in the context of panel clustering and $\cH$- and $\cH^2$-matrices, see \cite{Hac2015} and the references therein.

\subsection{Adaptive data sparse approximation} \label{sec:adaptive}

Let $\cP_{q,q}(\ATh)$ be the space of piecewise polynomials of degree $\le q$ in each variable over the mesh $\ATh$.
Since the adaptive hierarchical mesh is obtained by coarsening of the underlying uniform grid $\Th \times \Th$, we certainly have
\begin{align*}
    \cP_{q,q}(\ATh) \subset \cP_{q,q}(\Th \times \Th).
\end{align*}
Instead of a uniform approximation as used in \Cref{sec:uniform}, we now consider adaptive approximations $k_h \in \cP_{q,q}(\ATh)$ for the evaluation of \eqref{eq:abstract} or \eqref{eq:y_j^m_dense}. %
\begin{remark}
Let us assume for the moment that the kernel $k$ in \eqref{eq:integral} is  \emph{asymptotically smooth}, i.e., there exist constants  $c_1,c_2>0$, $r\in\RR$ such that
\begin{align} \label{eq:smooth}
\big|\partial_t^{\alpha}\partial_s^{\beta}k(t,s)\big|
\leq
c_1\frac{(\alpha+\beta)!}
{c_2^{{\alpha+\beta}}}
(t-s)^{r-\alpha-\beta}
\end{align}
for all $\alpha$, $\beta \ge 0$ and all $t \ne s$.
As shown in \cite{Bor2010,Hac2015}, adaptive approximations $k_h \in \cP_{q,q}(\ATh)$ can be constructed for asymptotically smooth kernels, which converge exponentially in $q$ in the farfield. 
As a consequence, the same level of accuracy required in \Cref{lem:abstract} can be achieved by adaptive approximations with much less degrees of freedom than by uniform approximations.
\end{remark}

\begin{remark}
It is not difficult to see that 
$\operatorname{dim}(\cP_{q,q}(\Th \times \Th)) = \cO(N^2q^2)$ 
while $\operatorname{dim}(\cP_{q,q}(\ATh)) = \cO(Nq^2)$.  
The adaptive hierarchical approximation thus is \emph{data-sparse} and leads to substantial savings in the memory required for storing the kernel approximation or its matrix respresentation \eqref{eq:algebraic2}; compare to \Cref{lem:H2memory} at the end of this section.
In addition, an appropriate reorganisation of the operations required for the matrix-vector product \eqref{eq:algebraic2} leads to a substantial reduction of computational complexity. Moreover, the fast evaluation also induces an automatic compression of the history of the data. 
\end{remark}

\subsection{Multilevel hierarchical basis}\label{sec:multilevel}

In order to obtain algorithms for the matrix-vector-multiplication \eqref{eq:algebraic} of quasi-optimal complexity, we require the following second fundamental ingredient.
Based on the multilevel hierarchy $I^{(n;\ell)}$ of time intervals and the translation invariance of the basis functions $\varphi_i^n =: \varphi_i^{(n;1)}$, we define a multilevel basis
\begin{align}\label{eq:multilevelbasis}
\varphi^{(n;\ell)}_i(t)
=
\begin{cases}
\sum_{j=0}^q A_{i,j}^{(1)}\varphi^{(2n-1;\ell-1)}_j(t),& t\in I^{(2n-1;\ell-1)},\\
\sum_{j=0}^q A_{i,j}^{(2)}\varphi^{(2n;\ell-1)}_j(t),& t\in I^{(2n;\ell-1)},
\end{cases}
\end{align}
for the spaces $\cP_q(I^{(n;\ell)})$, $\ell > 1$ appearing in the far field blocks of the approximation $k_h \in \cP_{q,q}(\ATh)$.
Let us note that by translation invariance, the coefficients $A_{i,j}^{(1)}$ and $A_{i,j}^{(2)}$ are independent of $n$ and $\ell$.

For each of the elements of the adaptive partition $\ATh$, we expand the kernel function as
\begin{align} \label{eq:kernelexpansion}
k_h(s,t)=\sum_{i=0}^q\sum_{j=0}^q k_{i,j}^{(m,n;\ell)}\varphi _i^{(m;\ell)}(s)\varphi _j^{(n;\ell)}(t),\qquad (s,t)\in I^{(m;\ell)}\times I^{(n;\ell)}.
\end{align}
For the computation of the farfield contributions in \eqref{eq:y_j^m_dense}, we can further split the integration domain into
\begin{align}\label{eq:farfieldsplitting}
[0,t^{m-2}]
=
\bigcup_{\ell=1}^{L(m)}\bigcup_{n=1}^{B(m;\ell)}I^{(P(m,n;\ell);\ell)}
\end{align}
with 
$L(m)=\lceil\log_2(m)\rceil-1$,
$B(m;\ell)=\bin(m)_\ell+1$,
and $P(m,n;\ell)=C(m;\ell)-n-1$.
Here $\bin(m)_\ell$ denotes the $\ell$th digit from behind of the binary representation of $m$ obtained by {\sc Matlab}'s \texttt{dec2bin} function. 
This partitioning of the integration domain exactly corresponds to the cells of a corresponding row in the adaptive mesh $\ATh$; see \Cref{fig:LBPillustration} for an illustration.
More precisely, $L(m)\in\NN$ describes the number of different coarsening levels involved in the $m$th row, $B(m;\ell)\in\{1,2\}$ corresponds to the number of intervals on each level, and $P(m,n;\ell)$ yields the indices of these intervals on level $\ell$.
\begin{figure}[ht!]
\centering
\input{figures/adaptive-mesh-L3-illustration}
\caption{Illustration of the meaning of the quantities $L(m)$, $B(m,\ell)$, and $P(m,n;\ell)$ arising in \eqref{eq:farfieldsplitting} for $m=14$ time steps.}
\label{fig:LBPillustration}
\end{figure}
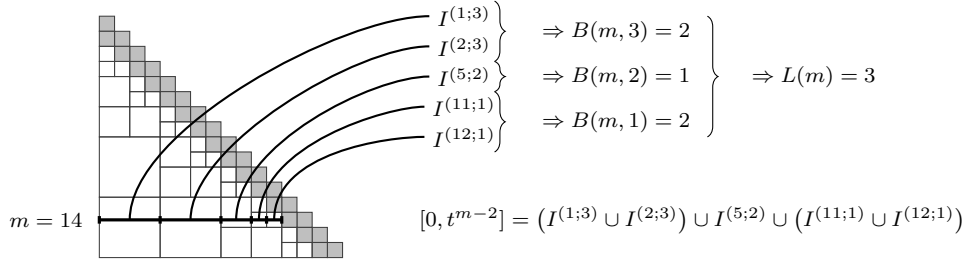
By inserting the splitting of the integration domain in \eqref{eq:farfieldsplitting}  into equation \eqref{eq:y_j^m_dense}, we obtain
\begin{align*}
y_h(t_j^m)
=
\sum_{\ell=1}^{L(m)}\sum_{n=1}^{B(m;\ell)}
\int_{I^{(P(m,n;\ell);\ell)}}k_h(t_j^m,s)f_h(s)\d s + \int_{t^{m-2}}^{t_j^m}k_h(t_j^m,s)f_h(s)\d s.
\end{align*}
We can now further insert the expansions \eqref{eq:kernelexpansion} for the kernel $k_h$ into the farfield integrals over $I^{(P(m,n;\ell);\ell)}$ to see that 
\begin{align*}
&\int_{I^{(P(m,n;\ell);\ell)}}k_h(t_j^m,s)f_h(s)\d s\\
&\hspace{2cm}=
\sum_{i=0}^q\varphi_i^{(C(n;\ell);\ell)}(t_j^m)\sum_{k=0}^q k_{i,k}^{(C(n;\ell),P(m,n;\ell);\ell)}g_k^{(P(m,n;\ell);\ell)},
\end{align*}
with auxiliary values
\begin{align*}
g_k^{(P(m,n;\ell);\ell)}
=
\int_{I^{(P(m,n;\ell);\ell)}}\varphi_k^{(P(m,n;\ell);\ell)}(s)f_h(s)\d s.
\end{align*}
By the recursive definition of  $\varphi^{(n;\ell)}$, the latter expression can be rewritten as
\begin{align*}
g_k^{(i;\ell)}
=
\int_{I^{(i;\ell)}}\varphi_k^{(i;\ell)}(s)f_h(s)\d s
=
\sum_{j=0}^q
\big(
A_{i,j}^{(1)}g_k^{(2i-1;\ell-1)}+A_{i,j}^{(2)}g_k^{(2i;\ell-1)}
\big)
\end{align*}
for $\ell > 1$ complemented with $g_k^{(i;1)} = g_k^i = \sum_{r=0}^p Q_{k,r} f^i_r$ as defined on the uniform grid in \Cref{sec:uniform}.
Evaluation of the recursion \eqref{eq:multilevelbasis} at time $t_j^m$ further yields
\begin{align*}
\varphi_i^{(C(n;\ell);\ell)}(t_j^m)=\sum_{k=0}^q A_{i,k}^{(B(m,\ell-1))}\varphi_k^{(C(n,\ell-1);\ell-1)}(t_j^m),
\end{align*}
such that we may define intermediate values
\begin{align*}
u_j^{(C(n;\ell);\ell)}
&=
\sum_{i=0}^q A_{i,j}^{(B(m;\ell))}u_i^{(C(m,\ell+1);\ell+1)}
+
\sum_{n=1}^{B(m;\ell)}\sum_{k=0}^q k_{i,k}^{(C(n;\ell),P(m,n;\ell);\ell)}g_k^{(P(m,n;\ell);\ell)}.
\end{align*}
The result of the integral \eqref{eq:y_j^m_dense} then is finally obtained by 
\begin{align*}
y_j^m=y_h(t_j^m) = \sum_{k=0}^q P_{j,k}u_k^{(m;1)}+z^m
\end{align*}
with nearfield contributions $z^m$ and projections $P_{j,k}$ as given in \eqref{eq:PijandQij} and \eqref{eq:PandQ}. 
The above derivations can be summarized as follows.

\begin{algorithm}[ht!]
\caption{\label{alg:fastconvolution}A fast and oblivious evolutionary algorithm.}
\begin{algorithmic}[1]
\For{$m=1,\ldots,N$}
\State $L_{\text{coarse}}=1+\lfloor\log_2(\texttt{bitxor}(m,m-1))\rfloor$
\For{$\ell=L_{\text{coarse}},\ldots,1$}
\If{$B(m;\ell)\neq B(m-1;\ell)$}
\State$\ttg^{(2;\ell)}=\ttg^{(1;\ell)}$
\If{$\ell>1$}
\State$\ttg^{(1;\ell)}=A^{(1)}\ttg^{(1;\ell-1)}+A^{(2)}\ttg^{(2;\ell-1)}$
\Else
\State$\ttg^{(1;\ell)}=Q\ttf^{(2)}$
\EndIf
\State Set $(\ttK_n)_{i,j}=k_{i,j}^{(C(n;\ell),P(m,n;\ell);\ell)}$ for $n\in\{1,B(m;\ell)\}$
\State$\ttu^\ell=\ttK_1\ttg^{(1;\ell)}$
\If{$B(m;\ell)=2$}
\State$\ttu^\ell=\ttu^\ell+\ttK_2\ttg^{(2;\ell)}$
\EndIf
\State$\ttu^\ell = \ttu^\ell + \big(A^{(B(m;\ell))}\big)^\top\ttu^{(\ell+1)}$
\EndIf
\EndFor
\State $\ttf^{(2)}=\ttf^{(1)}$
\State $\ttf^{(1)}_j=f(t_j^m)$, $j=0,\ldots,p$
\State $z^m=K^{m,m-1}\ttf^{(2)}+K^{m,m}\ttf^{(1)}$
\State $y^m=P\ttu^{(1)}+z^m$
\EndFor
\end{algorithmic}
\end{algorithm}

\noindent
\begin{remark}
The {\sc Matlab} function $\texttt{bitxor}(a,b)$ returns the integer generated by a bit-wise xor comparison of the binary representation of $a$ and $b$ in $\cO(1)$ complexity.
This allows to determine $L_{\text{coarse}}=\argmax_{k}\{B(m;k)\neq B(m-1;k)\}$ in $\cO(1)$ complexity in each step and is required for achieving a theoretical runtime of $\cO(N)$.
In actual implementations one may just set $L_{\text{coarse}}=L(m)$, without any notable difference in computation times.
Further note that only one value $u^{(n;\ell)}$ and two values of $g^{(n;\ell)}$ are required for each level $\ell$.
Moreover, at most two values of $f^n$ are required at any timestep. The required buffers are denoted by $\ttu^{(\ell)}$, $\ttf^{(i)}$, and $\ttg^{(i;\ell)}$, $i=1,2$, $\ell=1,2,3,\ldots$. 
The complexity of the overall algorithm is analyzed in detail in the next section.
\end{remark}

\begin{remark}
Readers familiar with $\cH^2$-matrices or the fast multipole method may notice that our algorithm is similar to the corresponding matrix-vector multiplications but with rearranged computations. In each time step, the algorithm checks for changes in the matrix partitioning structure compared to the previous time step. Then, starting from the coarsest level, an upward pass of one level is executed for all coarsened intervals of the farfield by applying the \emph{transfer} or \emph{multipole-to-multipole matrices}. Entities from the coarsened intervals are overwritten. Thereafter the farfield interactions are computed by applying the \emph{kernel} or \emph{multipole-to-local matrices} corresponding to the changed intervals and directly including the contributions of the next coarser level with a downward pass by appling \emph{transfer} or \emph{local-to-local matrices}. Finally, the nearfield contributions are applied.
\end{remark}

\subsection{Complexity estimates} 

In the following, we consider \Cref{alg:fastconvolution} for the evaluation of \eqref{eq:y_j^m_dense} with approximate kernel $k_h \in \cP_{q,q}(\ATh)$ and data $f_h \in \cP_p(\Th)$, 
and with $N=2^L$ denoting the number of time intervals in $\Th$. 
The assertions of the following two lemmas are well-known, see e.g.~\cite{Bor2010}, but their reasoning is simple and illustrative such that we repeat it for the convenience of the reader. %

\begin{lemma}\label{lem:evalcomplexity}
\Cref{alg:fastconvolution} can be executed in $\cO\big(N(p^2+q^2)\big)$ operations.
\end{lemma}
\begin{proof}
The algorithm rearranges the operations of a standard $\cH^2$-matrix-vector multiplication without adding any significant operations. We therefore simply estimate the complexity of the corresponding $\cH^2$-matrix-vector multiplication.
Let us first remark that the computation of $z^m$ in line 21 requires $\cO(p^2)$ operations in each time step. Second, on a given level $\ell$, we have to perform $\cO(2^\ell)$ applications of $A^{(1)}$ and $A^{(2)}$ in total for obtaining the $g^{(n;\ell)}$ from the ones on level $\ell-1$, see line 7. Similarly, $\cO(2^\ell)$ applications of $A^{(B(m;\ell))}$ in line 16 are in total required on level $\ell$ for the computation of the $u^{(n;\ell)}$ and $\cO(2^\ell)$ multiplications by $k^{(k,n;\ell)}$ need to be performed in lines 12 and 14. Finally, $\cO(N)$ values of $g^n=Qf^n$ and $P\ttu^{(1)}$ need to be computed in line 9 and line 22. Summing up yields
\begin{align*}
    \cO(Np^2)+3\cO(q^2)\sum_{\ell=1}^L\cO(2^{L-\ell}) + 2\cO(Npq)
    =
    \cO(Np^2)+\cO(2^L q^2) + \cO(Npq),
\end{align*}
and since $N=2^L$ Young's inequality yields the assertion. 
\hfill \qed
\end{proof}

\begin{lemma}\label{lem:H2memory}
The $\cH^2$-matrix representation $\ttK$ of the adaptive hierarchic approximation $k_h \in \cP_{q,q}(\ATh)$ can be stored in $\cO\big(N(p^2+q^2)\big)$ memory. 
If the kernel is of convolution type \eqref{eq:convolution_kernel}, then the memory cost reduces to $\cO\big(p^2+\log_2(N)q^2)\big)$. 
\end{lemma}
\begin{proof}
The proof for the adaptive approximation is similar to the previous lemma, with the $p^2$-related term arising from the nearfield and the $q^2$-related term from the farfield. For a kernel of convolution type, the hierarchical approximation provides a block Toeplitz structure, such that we only have to store $\cO(1)$ coefficient matrices per level for the farfield and $\cO(1)$ coefficient matrices for the nearfield.
\hfill \qed
\end{proof}

Let us finally also remark on the additional memory required during execution.

\begin{lemma}\label{lem:memory}
The active memory required for storing the data history required for \Cref{alg:fastconvolution} is bounded by $\cO(q \log_2N+p)$.
\end{lemma}
\begin{proof}
We require $O(1)$ vectors of length $p$ for the nearfield and at most two vectors $g^{(n;\ell)}$ of length $q$ on  $L=\log_2(N)$ levels for the farfield contributions. 
\hfill \qed
\end{proof}

\subsection{Summary} 
In this section, we discussed the adaptive hierarchical data sparse approximation for the dense system matrix $\ttK$ in \eqref{eq:algebraic} stemming from a uniform polynomial-based discretization of the Volterra-integral operators \eqref{eq:integral}. 
This approximation amounts to an $\cH^2$-matrix compression of the system matrix, leading to $\cO(N)$ storage complexity for general and $\cO(\log(N))$ storage complexity for convolution kernels. 
Using a multilevel basis representation on the hierarchy of the time intervals, we formulated a fast and oblivious evolutionary algorithm for the numerical evaluation of Volterra integrals \eqref{eq:integral}.
The overall complexity for computing the matrix-vector product in \eqref{eq:algebraic} is $\cO(N)$ and only $\cO(\log(N))$ memory is required to store the compressed history of the data. 
The algorithm is executed in an oblivious and evolutionary manner and can therefore be generalized immediately to integro-differential equations of the form \eqref{eq:integrodiff}. Moreover, knowledge of the number of time steps $N$ is not required prior to execution.


\section{Approximation of convolution operators}\label{sec:frequency}

While our algorithm for the fast evaluation of Volterra integral operators is based on a time domain approximation, 
in many interesting applications, see  \cite{Lub2004} for examples and references, 
the kernel function in \eqref{eq:volterra} is of convolution type
\begin{align}\label{eq:lubichconvolutionkernel}
k(t,s)=k(t-s)
\end{align}
and only accessible indirectly via its Laplace transform, i.e., the transfer function
\begin{align*}
\K(s)\isdef(\mathcal{L}k)(s)\isdef\int_0^\infty e^{-st}k(t)\d t,\quad s\in\CC.
\end{align*}
Let us note that, at least formally, the evaluation of the kernel function in the time domain can 
be achieved by the inverse Laplace transform
\begin{align}\label{eq:invlaplace}
k(t)=(\mathcal{L}^{-1} \K)(t)=\frac{1}{2\pi i}\int_{\Gamma}e^{t\lambda} \K(\lambda)\d\lambda,\quad t>0,
\end{align}
where $\Gamma$ is an appropriate contour connecting $-i\infty$ with $i\infty$; see \cite{ABHN2011} for details. 
To ensure the existence of the inverse Laplace transform, we require that 
\begin{align}
\label{eq:sector}
\text{
$\K(\lambda)$ is analytic in a sector $|\arg(\lambda-c)|<\varphi$, $\frac{\pi}{2}<\varphi<\pi$,}\\
\label{eq:Kdecay}
\text{
and $|\K(\lambda)|\leq M |\lambda|^{-\mu}$ for some fixed $M,\mu>0$,
}
\end{align}
and tacitly assume that the contour $\Gamma$ lies within the domain of analyticity of the function $\hat{k}$.
In this section, we show that with some minor modifications, \Cref{alg:fastconvolution} and our analysis are applicable 
also in this setting and we compare our method with the \emph{fast and oblivious convolution quadrature} 
of \cite{LS2002,SLL2006}.

\subsection{Approximation and numerical realization} 

As a first step, we show that the convolution kernel $k$ given implicitly by \eqref{eq:invlaplace} indeed satisfies the assumption \eqref{eq:smooth} on asymptotic smoothness. 
Thus an accurate adaptive hierarchical approximation as discussed in \Cref{sec:fo} is feasible.
\begin{lemma}\label{lem:isas}
Assume that $\K$ satisfies \eqref{eq:sector} and \eqref{eq:Kdecay}. Then $k$ as defined in \eqref{eq:invlaplace} is asymptotically smooth, i.e., it satisfies \eqref{eq:smooth} with $c_2=\sin(\varphi-\pi/2)$.
\end{lemma}
\begin{proof}
It is sufficient to consider the case $c=0$ in \eqref{eq:sector} and $\mu= 1$ in \eqref{eq:Kdecay}. Otherwise, we simply transform $\K(\lambda+c)=\cL(e^{-ct}k(t))(\lambda)$ and $k(t)=k_*^{(\mu-1)}(t)$ with $\K_*(\lambda)\isdef\cL(k_*)(\lambda)=|\lambda|^{\mu-1}\K(\lambda)$ for $\mu\neq 1$.
From \cite[Theorem 2.6.1]{ABHN2011}, also see \cite{Sov1979}, we deduce that $k$ has a holomorphic extension into the sector $|\arg(\lambda)|<\varphi-\pi/2$ with $\varphi$ as in \eqref{eq:sector}. Thus, the radius of convergence of the Taylor series of $k$ around $t_0\in(0,\infty)$ is given by $c_2 t_0$, with $c_2=\sin(\varphi-\pi/2)$ independent of $t_0$. This implies
\begin{align*}
\big|\partial_t^{\alpha}k(t)\big|
\leq
c_1
\frac{\alpha!}
{c_2^{\alpha}t^{\alpha}}
\end{align*}
for some constant $c_1>0$. Condition \eqref{eq:smooth} then follows by the chain rule. 
\hfill \qed
\end{proof}

For the construction of the adaptive approximation $k_h$, we can now proceed in complete analogy to \eqref{eq:split}, i.e., we split the convolution integral 
\begin{align}\label{eq:splitconvolution}
    y_h(t^n) = \int_0^{t^{n-2}} k_h(t^n,s) f_h(s) \d s+\int_{t^{n-2}}^{t^n} k_h(t^n,s) f_h(s) \d s
\end{align}
into a farfield and a nearfield contribution. 
The latter can be computed stably and efficiently with Lubich's convolution quadrature; see \cite{Lub1988,Lub1988a} for details. 

For the farfield contributions, we utilise 
the adaptive hierarchical approximation discussed in the previous section. 
If direct access to the kernel $k(t,s)=k(t-s)$ is available, \Cref{alg:fastconvolution} can be applied directly. 
If, on the other hand, only the transfer function $\K(s)$ is accessible, the values of $k_h(t,s)=k(t-s)$ can be computed by fast numerical Laplace inversion; see \cite{DW2015,LPS2006,Tal1979}. 
Here we follow the approach of \cite{LPS2006,SLL2006} which is based on hyperbolic contours of the form 
\begin{align}\label{eq:contour}
\gamma(\theta) = \mu(1-\sin(\alpha + i\theta)) + \sigma,\qquad\theta\in\RR,
\end{align}
with $0<\mu$, $0<\alpha<\pi/2-\varphi$, and $\sigma\in\RR$, such that the contour remains in the sector of analyticity \eqref{eq:sector} of $\K$.   
The discretisation of the contour integral \eqref{eq:invlaplace} by the trapezoidal rule with uniform step with $\tau$ yields
\begin{align}\label{eq:invlaplacequadrature}
k(t) \approx \sum_{r=-R}^R \frac{i\tau}{2\pi}e^{\gamma(\theta_r)t}\gamma'(\theta_r) \K(\gamma(\theta_r)), 
\end{align}
with $\theta_r=\tau r$. Given we are interested in $k(t)$ for $t\in [t_{\min},t_{\max}]$ and have fixed values for $\alpha$ and $\sigma$, suitable parameters $\tau$ and $\mu$ are given by
\[
\tau = a_\rho(\rho_{\text{opt}}),
\quad
\mu = \frac{2\pi \alpha R (1-\rho_{\text{opt}})}{t_{\max} a_\rho(\rho_{\text{opt}})},
\quad
\rho_{\text{opt}} = \argmin_{\rho\in (0,1)} \, \left(\eps\eps_R (\rho) ^{\rho-1} + \eps_R(\rho)^\rho\right),
\]
where $\eps$ is the machine precision and
\[ 
a_\rho (\rho) = \acosh\left(\frac{t_{\max}/t_{\min}}{(1 - \rho)\sin(\alpha)}\right),\quad \eps_R(\rho) = \exp\left(-\frac{2\pi \alpha R}{a_\rho(\rho)}\right),
\]
see \cite{LPS2006,SLL2006}. In our examples in \Cref{sec:numerics} we chose $\alpha = 3/16\pi$, $\sigma = 0$.
For error bounds concerning this approach we refer to \cite{LPS2006,SLL2006}.


\subsection{Comparison with fast and oblivious convolution quadrature}

Similar to \Cref{alg:fastconvolution}, the \emph{fast and oblivious convolution quadrature} (FOCQ) method of \cite{LS2002,SLL2006} is also based on a splitting \eqref{eq:splitconvolution} of the convolution integral into nearfield and farfield contributions, and the former can again be computed stably and efficiently with Lubich's convolution quadrature \cite{Lub1988,Lub1988a}. 

A different adaptive hierarchic approximation based on L-shaped cells is now used for the approximation in the farfield; see \Cref{fig:fastandoblivious}. 
\begin{figure}[ht!]
    \centering
    \input{figures/FOCQ-L-shapes-L1}
    \hspace*{1em}
    \input{figures/FOCQ-L-shapes-L2}
    \hspace*{1em}
    \input{figures/FOCQ-L-shapes-L3}
    \caption{Hierarchical partitions of fast and oblivious convolution quadrature~\cite{SLL2006}.
    \label{fig:fastandoblivious}}
\end{figure}
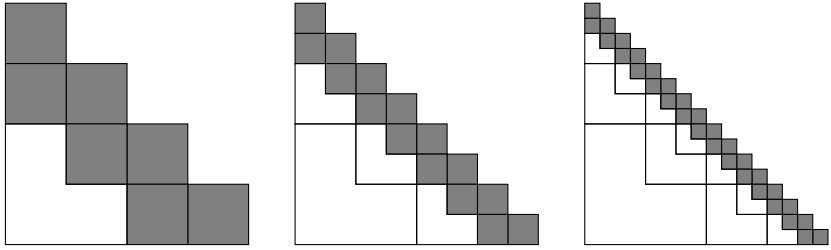
The farfield part of the integration domain for computing the entry $y_h(t^m)$ is then partitioned into
\begin{align*}
    [0,t^{m-2}]
    =
    \bigcup_{n=1}^{m-2}I^n
    =
    \bigcup_{\ell=1}^{L(m)}\bigcup_{n=1}^{B(m;\ell)}I^{(P(m,n;\ell);\ell)}
    =
    \bigcup_{\ell=1}^{L(m)}I_{\text{FOCQ},m}^{\ell}.
\end{align*}
Choosing an appropriate contour $\Gamma_{\ell}$, see \eqref{eq:contour}, and corresponding quadrature points $\theta_r^{(\ell)}$ for each farfield cell and using \eqref{eq:invlaplacequadrature} yields an approximation 
\begin{align}\label{eq:FOCQfarfield}
&\int _{I_{\text{FOCQ},m}^\ell}k\big(t^m,s\big) f(s)\d s\\
\approx{}&\frac{i\tau}{2\pi}\sum_{r=-R}^R\K\big(\gamma(\theta_r^{(\ell)})\big)\gamma'\big(\theta_r^{(\ell)}\big)e^{\gamma(\theta_r^{(\ell)})(t^m-b^{(\ell)})}\underbrace{\int _{I_{\text{FOCQ},n}^\ell}e^{\gamma(\theta_r^{(\ell)})(b^{(\ell)}-s)}f(s)\d s}_{=z(c^{(\ell)};b^{(\ell)},\gamma(\theta_r^{(\ell)}))}, \notag
\end{align}
with $b^{(\ell)}=\min I_{\text{FOCQ},m}^\ell$ and $c^{(\ell)}=\max I_{\text{FOCQ},m}^\ell$. The values $z(c^{(\ell)};b^{(\ell)},\gamma(\theta_r^{(\ell)}))$ can be computed by numerically solving the ordinary differential equation
\begin{align}\label{eq:lubichode}
\frac{d}{dt}z(t;s,\gamma) = \gamma z(t;s,\gamma) + f(t),\qquad z(s;s,\gamma) = 0
\end{align}
with appropriate values $s=b^{(\ell)}$ and $\gamma=\gamma(\theta_r^{(\ell)})$.  
Thus, the fast and oblivious convolution quadrature provides an approximation of the convolution matrix by solving an auxiliary set of $(2R+1)L$ differential equations. 
In order to obtain an oblivious algorithm it is crucial that the solution of each differential equation is updated in each time step, i.e., the compressed convolution matrix must be evaluated from \emph{left to right}; see \cite{LS2002,SLL2006} for details.

The connection to our approach is revealed upon noticing that the compression approach underlying the fast and oblivious convolution quadrature actually implements a low-rank approximation in each of the farfields L-shaped blocks, i.e.,
\begin{align*}
k(t,s)\approx{}&\sum_{r=-R}^R\bigg(\frac{i\tau}{2\pi}e^{\gamma(\theta_r^{(\ell)})(t-b^{(\ell)})}\K\big(\gamma\big(\theta_r^{(\ell)}\big)\big)\gamma'\big(\theta_r^{(\ell)}\big)\bigg) e^{\gamma(\theta_r^{(\ell)})(b^{(\ell)}-s)}\\
={}&\sum_{r=-R}^RU(t,\theta_r^{(\ell)})V(s,\theta_r^{(\ell)}).
\end{align*}
The corresponding farfield approximation \eqref{eq:FOCQfarfield} thus effectively reads
\begin{align*}
\int _{I_{\text{FOCQ},m}^\ell}k\big(t^{n},s\big) f(s)\d s
\approx{}&
\sum_{r=-R}^RU\big(t,\theta_r^{(\ell)}\big)\int _{I_{\text{FOCQ},m}^\ell}V\big(s,\theta_r^{(\ell)}\big)f(s)\d s\\
={}&
\sum_{r=-R}^RU\big(t,\theta_r^{(\ell)}\big)z\big(c^{(\ell)},b^{(\ell)},\theta_r^{(\ell)}\big),
\end{align*}
which can be understood as a low-rank matrix-vector product realized implicitly by the numerical solution of a differential equation.
Since the partitioning depicted in \Cref{fig:fastandoblivious} can  easily be refined to an adaptive partitioning as in \Cref{fig:adaptive}, the fast and oblivious convolution quadrature can be interpreted as a particular case of an $\cH$-matrix approximation with a particular realisation of the $\cH$-matrix-vector product. 
The $\cO(\log(N))$ memory cost and $\cO(N\log(N))$ complexity of the algorithm can then be immediately deduced from $\cH$-matrix literature \cite{Bor2010,Hac2015}. 

As mentioned in the introduction, the algorithm proposed in \Cref{sec:adaptive}, with the modifications discussed above, is based on an $\cH^2$-matrix approximation and leads to a better complexity of $\cO(N)$.
It is also clear that the number of quadrature points for the numerical Laplace inversion determines the ranks of the far-field approximations for the $\cH$-matrix approximations, which allows for an improved understanding of the approximation error in terms of the approximation order.

\section{Numerical examples}\label{sec:numerics}

In the following we present a series of numerical examples to illustrate and verify the capabilities of our novel algorithms. The experiments are performed in {\sc Matlab}, in which we implemented our new algorithm as well as a version of the fast-and-oblivious convolution-quadrature algorithm of \cite{LPS2006,LS2002} for reference. 
Although our new algorithm performed considerably faster in all of the following examples, we do not present a detailed comparison.

In accordance with the $\cH$- and $\cH^2$-literature, we require the farfield cells in our implementation to be at least $n_{\min}\times n_{\min}=16\times 16$ times larger than the nearfield cells. 
This simply means that the cells in \Cref{fig:adaptive} represent $16 \times 16$ blocks of fine-grid cells.
Following \cite[Chapter 2]{Bru2004}, we choose Radau-IIA collocation methods of stage $p=1,2,3$, for the discretization of the Volterra integral operators, which is exactly the scheme used for the approximation as outlined in \Cref{sec:uniform} and the beginning of \Cref{sec:uniform2}. This fixes the approximation spaces $\cP_p(\Th)$ for the solution $y_h$ and the data $f_h$.
The error of this discretization scheme is given by
\begin{align}\label{eq:brunner}
e_h\defis\max_{t_i\in\Th}|y(t_i)-y_h(t_i)|\leq Ch^{2p-1}
\end{align}
for smooth data $f\in C^{2p-1}([0,T])$ and kernel $k\in C^{2p-1}(\{(t,s)\colon 0\leq t\leq s\leq T\})$; see \cite[Chapter 2]{Bru2004} for details. 
For convolution kernels $k(t-s)$, we utilize Lubich's convolution quadrature \cite{Lub1988,Lub1988a} in the near field and the same strategy as above in the farfield. 
Again, the Radau-IIA method is used as the underlying integration scheme. discretization scheme in the near field is that of. This allows to estimate the approximation error by
\begin{align}\label{eq:lubicherror}
e_h\defis\max_{t_i\in\Th}|y(t_i)-y_h(t_i)|\leq C(h^{2p-1}+h^{p+1+\mu})
\end{align}
for transfer functions satisfying \eqref{eq:sector}--\eqref{eq:Kdecay} and data $f\in C^{2p-1}([0,T])$; see \cite{LO1993} for details. 
The parameter $\mu$ is related to the singularity of the kernel $k(t-s)$ at $t=s$.  

\subsection{Variation of constants formula}
\label{sec:variation_of_constants}

The first example is dedicated to the solution of the differential equation
\begin{align}
y'(t)&=-2t y(t)+5\cos(5t), \qquad t\in(0,10], \label{eq:lpe1}\\
y(0)&=2. \label{eq:lpe2}
\end{align}
By the variation of constants formula, the solution can be expressed as 
\begin{align} \label{eq:lpe3}
y(t)= 2e^{-t^2} + 5\int_0^t e^{s^2-t^2}\cos(5s)\d s. 
\end{align}
Let us note that the integral kernel $k(t,s)=e^{s^2-t^2}$ satisfies the asymptotic smoothness assumption \eqref{eq:smooth}, but is not of convolution type $k(t-s)$.
To obtain a reference solution, we solve \eqref{eq:lpe1}--\eqref{eq:lpe2} numerically with a 3-stage Radau-IIA method and with $N_\infty=2^{19}$ time steps.
For the numerical solution of \eqref{eq:lpe3}, we then employ \Cref{alg:fastconvolution} with polynomial degree $q=16$ for the kernel approximation and various degrees $p$ for the approximation of the data $f$ and the solution $y$.
\begin{figure}[ht]
\centering
\begin{subfigure}[hb]{.42\textwidth}
\centering
\input{plots/dvc-convergence}
\label{fig:dvc_convergence}
\end{subfigure}
\qquad
\begin{subfigure}[hb]{.42\textwidth}
\centering
\input{plots/dvc-times}
\label{fig:dvc_times}
\end{subfigure}
\caption{Approximation errors (left) and computation times (right) for the variation of constants formula example of \Cref{sec:variation_of_constants}.
\label{fig:lpe}}
\end{figure}
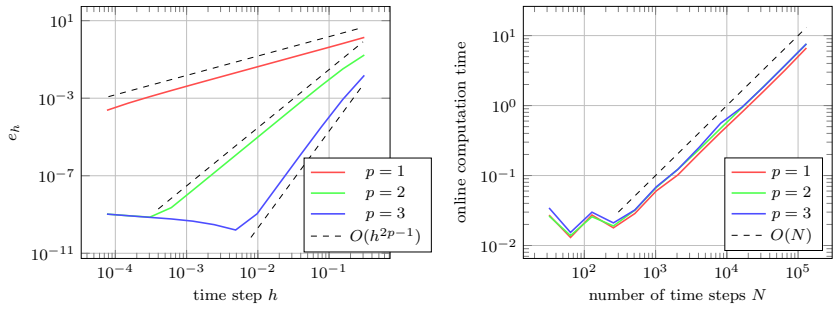
The left plot of \Cref{fig:lpe} illustrates that we indeed reach the theoretical convergence rates predicted by \eqref{eq:brunner} up to a tolerance of around $10^{-10}$ at which numerical noise begins to dominate.
From the right plot one can immediately deduce the linear complexity of the algorithm.

\subsection{Nonlinear Volterra integral equation} \label{sec:nonlinearvolterra}

We continue with a second test example taken from \cite{SLL2006}, in which we consider the nonlinear Volterra integral equation   
\begin{align*}
u(t) = -\int_0^t \frac{(u(\tau) - \sin(\tau))^3}{\sqrt{\pi(t-\tau)}} \d\tau,\qquad t\in[0,60].
\end{align*}
In this example, the convolution kernel $k(t-s) = 1/\sqrt{\pi (t-s)}$ is of convolution type, with Laplace transform given by $\K(\lambda)=1/\sqrt{\lambda}$. 
The evaluation of the integral kernel $k(t-s)$ is realized via numerical inverse Laplace transforms with $R=15$ quadrature points and the kernel is approximated by piecewise polynomials of degree $q=8$ in the farfield; see \Cref{sec:frequency}.
Since the data $f(t,u)=(u-\sin(t))^3$ here depend on $u$, a nonlinear equation must be solved for each time step, for which we employ a Newton method up to a tolerance of $10^{-12}$. 
As a reference solution for computing the errors, we here simply take the numerical solution computed on a finer grid.
\begin{figure}[ht]
\centering
\begin{subfigure}[hb]{.42\textwidth}
\centering
\input{plots/nvi-convergence}
\label{fig:nvi_convergence}
\end{subfigure}
\qquad
\begin{subfigure}[hb]{.42\textwidth}
\centering
\input{plots/nvi-times}
\label{fig:nvi_times}
\end{subfigure}
\caption{Approximation errors (left) and computation times (right) for the nonlinear Volterra integral equation example of \Cref{sec:nonlinearvolterra}.
\label{fig:nvi}
}
\end{figure}
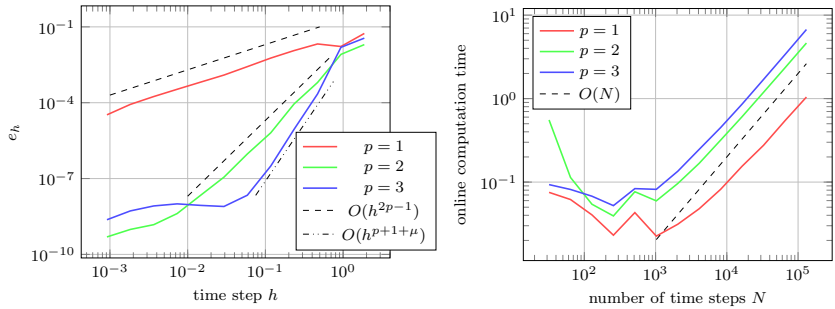
The results of \Cref{fig:nvi} again clearly show the predicted convergence rates up to the point where numerical noise begins to dominate, see
\eqref{eq:lubicherror} with $\mu=1/2$, and the linear complexity of \Cref{alg:fastconvolution}.


\subsection{Fractional diffusion with transparent boundary conditions}\label{sec:fractionaldiffusion}

As a last example, which is taken from \cite{CLP2006,SLL2006}, we consider the one-dimensional fractional diffusion equation
\begin{align*}
u(x,t) = u_0(x) + \int _0^t\frac{(t-\tau)^{\alpha-1}}{\Gamma(\alpha)}\Delta_{x}u(x,\tau)\d\tau + g(x,t),\qquad x\in\RR,~t\in\RR_{>0},
\end{align*}
with $\alpha=2/3$, $u(x,\cdot)\to 0$ for $|x|\to\infty$ and $g(x,0)=0$.
For the computations, we restrict the spatial domain to $x\in(-a,a)$, assume that $u_0$ and $g$ have support in $[-a,a]$, and impose transparent boundary conditions on $x=\pm a$ which read
\begin{align*}
u(x,t)=-\int _0^t\frac{(t-\tau)^{\alpha/2-1}}{\Gamma(\alpha/2)}\partial_{\bfn}u(x,\tau)\d\tau, \qquad x=\pm a;
\end{align*}
we refer to \cite{Hag1999,SLL2006} for further details on the model. 
The Laplace transform of the convolution kernel $k(t-s)=(t-s)^{\alpha-1}/\Gamma(\alpha)$ is here given by $\K(\lambda)=1/\lambda^{\alpha}$.

For the spatial discretisation we employ a finite difference scheme on an equidistant mesh $x_i=i\tau$, $\tau=a/M$, $i=-M,\ldots,M$ and use second-order finite differences within the domain and central differences to approximate the normal derivative on the boundary; see \cite{CLP2006,SLL2006}. 
For the time discretisation we employ the frequency domain version of our algorithm  with $R=30$ quadrature points for the inverse Laplace transform and polynomial degree $q=16$ for the farfield interpolation. 
We note that two different convolution quadratures are required, one for the fractional derivative in $(-a,a)$ involving $\alpha$ and one for the fractional derivative of the boundary values, involving $\alpha/2$.

For the space discretisation we consider a fixed mesh with $M=10^4$ which is fine enough to let the error of the time discretisation dominate. 
As a reference solution we take the method with order $p=3$ on a finer mesh.
\begin{figure}[ht]
\centering
\begin{subfigure}[hb]{.42\textwidth}
\centering
\input{plots/fd-convergence}
\label{fig:fd_convergence}
\end{subfigure}
\qquad\quad
\begin{subfigure}[hb]{.42\textwidth}
\centering
\input{plots/fd-times}
\label{fig:fd_times}
\end{subfigure}
\caption{Convergence plot (left) and computation times (right) for the fractional diffusion problem with transparent boundary conditions.\label{fig:fd}}
\end{figure}
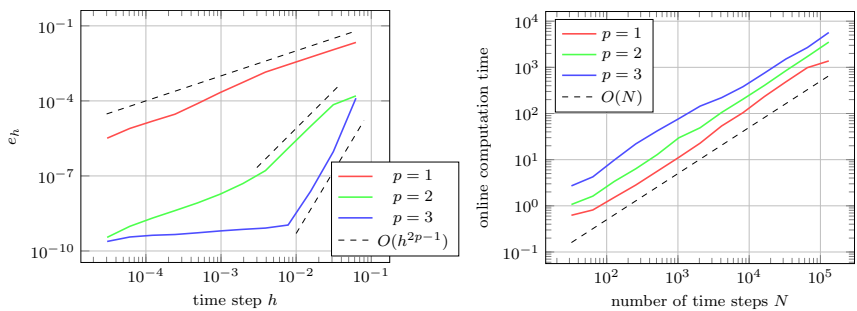
Let us note that due to the lack of temporal smoothness of the solution at $t=0$, we cannot expect the full order of convergence as predicted by \eqref{eq:lubicherror};  we refer to \cite[Section 1]{CLP2006} for further discussion. 
In \Cref{fig:fd}, we however still observe a very good approximation in the pre-asymptotic phase and a substantial improvement in accuracy until the numerical noise level is reached when using higher approximation order $p$. As predicted by the theory, the computation times again increase linearly in the number of time steps.  
In the numerical tests, identical results were obtained for direct evaluation of $k(t,s)$ and evaluation of the kernel via inverse Laplace transforms, which indicates that the main approximation error comes from the adaptive hierarchical approximation.

{\footnotesize
\section*{Acknowledgement}
The authors would like to thank the first reviewer for an exceptionally careful reading and valuable comments and remarks.

\section*{Funding}
Support by the German Science Foundation (DFG) via TRR~146 (project~C3), TRR~154 (project~C04), and SPP~2256 (project Eg-331/2-1) is gratefully acknowledged.
}

\bibliographystyle{plain} 
\bibliography{bibliography}
\end{document}

%% file: figures/uniform-mesh-L1.tex
\begin{tikzpicture}[scale=0.1]
\foreach \x in {0,...,3} {
\draw[fill=gray] (8*\x,24-8*\x) rectangle (8+8*\x,32-8*\x);
}
\foreach \x in {0,...,2} {
\draw[fill=gray] (8*\x,16-8*\x) rectangle (8+8*\x,24-8*\x);
}
\foreach \x in {0,...,1} {
\draw (8*\x,8-8*\x) rectangle (8+8*\x,16-8*\x);
}
\draw (0,0) rectangle (8,8);
\end{tikzpicture}

%% file: figures/uniform-mesh-L2.tex
\begin{tikzpicture}[scale=0.1]
\foreach \x in {0,...,7} {
\draw[fill=gray] (4*\x,28-4*\x) rectangle (4+4*\x,32-4*\x);
}
\foreach \x in {0,...,6} {
\draw[fill=gray] (4*\x,24-4*\x) rectangle (4+4*\x,28-4*\x);
}
\foreach \x in {0,...,5} {
  \pgfmathsetmacro{\z}{5-\x}
  \foreach \y in {0,...,\z} {
     \draw (4*\x,4*\y) rectangle (4+4*\x,4+4*\y);
 }
}
\end{tikzpicture}

%% file: figures/uniform-mesh-L3.tex
\begin{tikzpicture}[scale=0.1]
\foreach \x in {0,...,15} {
\draw[fill=gray] (2*\x,30-2*\x) rectangle (2+2*\x,32-2*\x);
}
\foreach \x in {0,...,14} {
\draw[fill=gray] (2*\x,28-2*\x) rectangle (2+2*\x,30-2*\x);
}
\foreach \x in {0,...,13} {
  \pgfmathsetmacro{\z}{13-\x}
  \foreach \y in {0,...,\z} {
     \draw (2*\x,2*\y) rectangle (2+2*\x,2+2*\y);
 }
}
\end{tikzpicture}

%% file: figures/mesh-1d.tex
\begin{tikzpicture}[xscale=1,yscale=0.4]

\foreach \x in {1,...,4} {
\draw (-1,2*\x-1.5) node {$\ell = \x$};
}

\foreach \x in {0,...,7} {
\draw (\x,0) rectangle (\x+1,1);
\pgfmathsetmacro\result{\x+1}
\draw (\x+0.5,0.5) node {$I^{(\pgfmathprintnumber{\result};1)}$};
}

\foreach \x in {0,...,3} {
\draw (2*\x,2) rectangle (2*\x+2,3);
\pgfmathsetmacro\result{\x+1}
\draw (2*\x+1,2.5) node {$I^{(\pgfmathprintnumber{\result};2)}$};
\draw[->] (2*\x+0.5,1) -- (2*\x+0.75,2);
\draw[->] (2*\x+1.5,1) -- (2*\x+1.25,2);
}

\foreach \x in {0,...,1} {
\draw (4*\x,4) rectangle (4*\x+4,5);
\pgfmathsetmacro\result{\x+1}
\draw (4*\x+2,4.5) node {$I^{(\pgfmathprintnumber{\result};3)}$};
\draw[->] (4*\x+1,3) -- (4*\x+1.5,4);
\draw[->] (4*\x+3,3) -- (4*\x+2.5,4);
}

\foreach \x in {0,...,0} {
\draw (8*\x,6) rectangle (8*\x+8,7);
\pgfmathsetmacro\result{\x+1}
\draw (8*\x+4,6.5) node {$I^{(\pgfmathprintnumber{\result};4)}$};
\draw[->] (8*\x+2,5) -- (8*\x+3,6);
\draw[->] (8*\x+6,5) -- (8*\x+5,6);
}


%
%
\end{tikzpicture}

%% file: figures/adaptive-mesh-L1.tex
\begin{tikzpicture}[scale=0.1]

\foreach \x in {0,...,3} {
\draw[fill=gray] (8*\x,24-8*\x) rectangle (8+8*\x,32-8*\x);
}
\foreach \x in {0,...,2} {
\draw[fill=gray] (8*\x,16-8*\x) rectangle (8+8*\x,24-8*\x);
}

\foreach \x in {0,...,1} {
\draw (8*\x,8-8*\x) rectangle (8+8*\x,16-8*\x);
}
\draw (0,0) rectangle (8,8);

\end{tikzpicture}

%% file: figures/adaptive-mesh-L2.tex
\begin{tikzpicture}[scale=0.1]

\foreach \x in {0,...,7} {
\draw[fill=gray] (4*\x,28-4*\x) rectangle (4+4*\x,32-4*\x);
}
\foreach \x in {0,...,6} {
\draw[fill=gray] (4*\x,24-4*\x) rectangle (4+4*\x,28-4*\x);
}

\foreach \x in {0,...,5} {
\draw (4*\x,20-4*\x) rectangle (4+4*\x,24-4*\x);
}
\foreach \x in {0,...,2} {
\draw (8*\x,16-8*\x) rectangle (4+8*\x,20-8*\x);
}

\foreach \x in {0,...,1} {
\draw (8*\x,8-8*\x) rectangle (8+8*\x,16-8*\x);
}
\draw (0,0) rectangle (8,8);

\end{tikzpicture}

%% file: figures/adaptive-mesh-L3.tex
\begin{tikzpicture}[scale=0.1]

\foreach \x in {0,...,15} {
\draw[fill=gray] (2*\x,30-2*\x) rectangle (2+2*\x,32-2*\x);
}
\foreach \x in {0,...,14} {
\draw[fill=gray] (2*\x,28-2*\x) rectangle (2+2*\x,30-2*\x);
}

\foreach \x in {0,...,13} {
\draw (2*\x,26-2*\x) rectangle (2+2*\x,28-2*\x);
}
\foreach \x in {0,...,6} {
\draw (4*\x,24-4*\x) rectangle (2+4*\x,26-4*\x);
}

\foreach \x in {0,...,5} {
\draw (4*\x,20-4*\x) rectangle (4+4*\x,24-4*\x);
}
\foreach \x in {0,...,2} {
\draw (8*\x,16-8*\x) rectangle (4+8*\x,20-8*\x);
}

\foreach \x in {0,...,1} {
\draw (8*\x,8-8*\x) rectangle (8+8*\x,16-8*\x);
}
\draw (0,0) rectangle (8,8);

\end{tikzpicture}

%% file: figures/adaptive-mesh-L3-illustration.tex
\begin{tikzpicture}[scale=0.1]

\foreach \x in {0,...,15} {
\draw[fill=lightgray,draw=darkgray] (2*\x,30-2*\x) rectangle (2+2*\x,32-2*\x);
}
\foreach \x in {0,...,14} {
\draw[fill=lightgray,draw=darkgray] (2*\x,28-2*\x) rectangle (2+2*\x,30-2*\x);
}

\foreach \x in {0,...,13} {
\draw[draw=darkgray] (2*\x,26-2*\x) rectangle (2+2*\x,28-2*\x);
}
\foreach \x in {0,...,6} {
\draw[draw=darkgray] (4*\x,24-4*\x) rectangle (2+4*\x,26-4*\x);
}

\foreach \x in {0,...,5} {
\draw[draw=darkgray] (4*\x,20-4*\x) rectangle (4+4*\x,24-4*\x);
}
\foreach \x in {0,...,2} {
\draw[draw=darkgray] (8*\x,16-8*\x) rectangle (4+8*\x,20-8*\x);
}

\foreach \x in {0,...,1} {
\draw[draw=darkgray] (8*\x,8-8*\x) rectangle (8+8*\x,16-8*\x);
}
\draw[draw=darkgray] (0,0) rectangle (8,8);

\draw (-7,5) node {\small $m=14$};

\draw[very thick] (0,5) -- (24,5);
\draw[very thick] (0,4.5) -- (0,5.5);
\draw[very thick] (8,4.5) -- (8,5.5);
\draw[very thick] (16,4.5) -- (16,5.5);
\draw[very thick] (20,4.5) -- (20,5.5);
\draw[very thick] (22,4.5) -- (22,5.5);
\draw[very thick] (24,4.5) -- (24,5.5);

\node (I31) at (48,32){\small $I^{(1;3)}$} ;
\node (I32) at (48,28){\small $I^{(2;3)}$} ;
\node (I25) at (48,24){\small $I^{(5;2)}$} ;
\node (I111) at (48,20){\small $I^{(11;1)}$} ;
\node (I112) at (48,16){\small $I^{(12;1)}$} ;

\draw[decorate,decoration={brace,amplitude=4}] (52,34) to (52,26);
\draw[thick] (4,5) .. controls +(up:12) and +(left:16) .. (I31);
\draw[thick] (12,5) .. controls +(up:10) and +(left:14) .. (I32);
\draw[decorate,decoration={brace,amplitude=4}] (52,26) to (52,22);
\draw[thick] (18,5) .. controls +(up:10) and +(left:12) .. (I25);
\draw[decorate,decoration={brace,amplitude=4}] (52,22) to (52,14);
\draw[thick] (21,5) .. controls +(up:10) and +(left:10) .. (I111);
\draw[thick] (23,5) .. controls +(up:9) and +(left:10) .. (I112);

\node (Bm3) at (68,30) {\small $\Rightarrow B(m,3)=2$};
\node (Bm2) at (68,24) {\small $\Rightarrow B(m,2)=1$};
\node (Bm1) at (68,18) {\small $\Rightarrow B(m,1)=2$};

\draw[decorate,decoration={brace,amplitude=4}] (80,32) to (80,16);
\node (L3) at (94,24) {\small $\Rightarrow L(m)=3$};

\draw (78,5) node {\small $[0,t^{m-2}]=\big(I^{(1;3)}\cup I^{(2;3)}\big)\cup I^{(5;2)}\cup \big(I^{(11;1)}\cup I^{(12;1)}\big)$};
\end{tikzpicture}

%% file: figures/FOCQ-L-shapes-L1.tex
\begin{tikzpicture}[scale=0.4]

\foreach \x in {0,...,3} {
\draw[fill=gray] (2*\x,30-2*\x) rectangle (2+2*\x,32-2*\x);
}
\foreach \x in {0,...,2} {
\draw[fill=gray] (2*\x,28-2*\x) rectangle (2+2*\x,30-2*\x);
}

\foreach \x in {0,...,0} {
\draw (4*\x,24-4*\x) -- (4+4*\x,24-4*\x) -- (4+4*\x,26-4*\x) -- (2+4*\x,26-4*\x) -- (2+4*\x,28-4*\x) -- (4*\x,28-4*\x) -- cycle;
}



\end{tikzpicture}

%% file: figures/FOCQ-L-shapes-L2.tex
\begin{tikzpicture}[scale=0.2]

\foreach \x in {0,...,7} {
\draw[fill=gray] (2*\x,30-2*\x) rectangle (2+2*\x,32-2*\x);
}
\foreach \x in {0,...,6} {
\draw[fill=gray] (2*\x,28-2*\x) rectangle (2+2*\x,30-2*\x);
}

\foreach \x in {0,...,2} {
\draw (4*\x,24-4*\x) -- (4+4*\x,24-4*\x) -- (4+4*\x,26-4*\x) -- (2+4*\x,26-4*\x) -- (2+4*\x,28-4*\x) -- (4*\x,28-4*\x) -- cycle;
}

\foreach \x in {0,...,0} {
\draw (8*\x,16-8*\x) -- (8+8*\x,16-8*\x) -- (8+8*\x,20-8*\x) -- (4+8*\x,20-8*\x) -- (4+8*\x,24-8*\x) -- (8*\x,24-8*\x) -- cycle;
}


\end{tikzpicture}

%% file: figures/FOCQ-L-shapes-L3.tex
\begin{tikzpicture}[scale=0.1]

\foreach \x in {0,...,15} {
\draw[fill=gray] (2*\x,30-2*\x) rectangle (2+2*\x,32-2*\x);
}
\foreach \x in {0,...,14} {
\draw[fill=gray] (2*\x,28-2*\x) rectangle (2+2*\x,30-2*\x);
}

\foreach \x in {0,...,6} {
\draw (4*\x,24-4*\x) -- (4+4*\x,24-4*\x) -- (4+4*\x,26-4*\x) -- (2+4*\x,26-4*\x) -- (2+4*\x,28-4*\x) -- (4*\x,28-4*\x) -- cycle;
}

\foreach \x in {0,...,2} {
\draw (8*\x,16-8*\x) -- (8+8*\x,16-8*\x) -- (8+8*\x,20-8*\x) -- (4+8*\x,20-8*\x) -- (4+8*\x,24-8*\x) -- (8*\x,24-8*\x) -- cycle;
}

\draw (0,0) -- (16,0) -- (16,8) -- (8,8) -- (8,16) -- (0,16) -- cycle;

\end{tikzpicture}

%% file: plots/dvc-convergence.tex
\begin{tikzpicture}[scale = 0.75]
\begin{loglogaxis}[
    xlabel={time step $h$},
    ylabel={$e_h$},
    grid=major,
    legend entries={$p=1$,$p=2$,$p=3$, $O(h^{2p-1})$},
    legend style={at={(1.135,0.4)}}
]
\addplot[mark=none, red!70, thick] coordinates {
(0.312500000000000,1.369822673744032)
(0.156250000000000,0.668715313877263)
(0.078125000000000,0.331693081958774)
(0.039062500000000,0.166062142926573)
(0.019531250000000,0.082637928628005)
(0.009765625000000,0.041204522652333)
(0.004882812500000,0.020539790209097)
(0.002441406250000,0.010222233462522)
(0.001220703125000,0.005067157013725)
(0.000610351562500,0.002490614427916)
(0.000305175781250,0.001202711752165)
(0.000152587890625,0.000559132364294)
(0.000076293945312,0.000238205198068)
};
\addplot[mark=none, green!70, thick] coordinates {
(0.312500000000000,0.169392341224405)
(0.156250000000000,0.032499608302425)
(0.078125000000000,0.004557191000023)
(0.039062500000000,0.000590053977974)
(0.019531250000000,0.000073690588495)
(0.009765625000000,0.000009179525763)
(0.004882812500000,0.000001144589456)
(0.002441406250000,0.000000142873497)
(0.001220703125000,0.000000017854584)
(0.000610351562500,0.000000002259280)
(0.000305175781250,0.000000000731205)
(0.000152587890625,0.000000000865029)
(0.000076293945312,0.000000001048062)
};
\addplot[mark=none, blue!70, thick] coordinates {
(0.312500000000000,0.015104694620772)
(0.156250000000000,0.000877220401696)
(0.078125000000000,0.000033356786457)
(0.039062500000000,0.000001102733895)
(0.019531250000000,0.000000034734007)
(0.009765625000000,0.000000001087381)
(0.004882812500000,0.000000000158703)
(0.002441406250000,0.000000000299614)
(0.001220703125000,0.000000000454309)
(0.000610351562500,0.000000000589663)
(0.000305175781250,0.000000000712974)
(0.000152587890625,0.000000000862844)
(0.000076293945312,0.000000001047379)

};
\addplot [dashed,domain=0.00008:0.3, samples=100]{15*x};
\addplot [dashed,domain=0.0004:0.3, samples=100]{30*x^3};
\addplot [dashed,domain=0.008:0.3, samples=100]{2*x^5};
\end{loglogaxis}
\end{tikzpicture}

%% file: plots/dvc-times.tex
\begin{tikzpicture}[scale = 0.75]
\begin{loglogaxis}[
    xlabel={number of time steps $N$},
    ylabel={online computation time},
    grid=major,
    legend entries={$p=1$,$p=2$,$p=3$, $O(N)$},
    legend pos=south east
]
\addplot[mark=none, red!70, thick] coordinates {
(32,     0.027195)
(64,     0.013032)
(128,    0.027283)
(256,    0.017883)
(512,    0.028639)
(1024,   0.060355)
(2048,   0.101658)
(4096,   0.206940)
(8192,   0.414129)
(16384,  0.806738)
(32768,  1.596351)
(65536,  3.206470)
(131072, 6.656376)
};
\addplot[mark=none, green!70, thick] coordinates {
(32,     0.026424)
(64,     0.013866)
(128,    0.025768)
(256,    0.019032)
(512,    0.031962)
(1024,   0.067440)
(2048,   0.122255)
(4096,   0.236062)
(8192,   0.472374)
(16384,  0.934831)
(32768,  1.864681)
(65536,  3.722699)
(131072, 7.602634)
};
\addplot[mark=none, blue!70, thick] coordinates {
(32,     0.034433)
(64,     0.015367)
(128,    0.029937)
(256,    0.021045)
(512,    0.032189)
(1024,   0.069012)
(2048,   0.122123)
(4096,   0.253296)
(8192,   0.558383)
(16384,  0.950371)
(32768,  1.853228)
(65536,  3.752174)
(131072, 7.631583)
};
\addplot [dashed,domain=300:131072, samples=100]{0.0001*x};
%
\end{loglogaxis}
\end{tikzpicture}

%% file: plots/nvi-convergence.tex
\begin{tikzpicture}[scale = 0.75]
\begin{loglogaxis}[
    xlabel={time step $h$},
    ylabel={$e_h$},
    grid=major,
    legend entries={$p=1$,$p=2$,$p=3$, $O(h^{2p-1})$, ,$O(h^{p+1+\mu})$},
    legend style={at={(1.15,0.5)}}
]
\addplot[mark=none, red!70, thick] coordinates {
(1.875000000000000,0.054972159499600)
(0.937500000000000,0.016742426226874)
(0.468750000000000,0.021092180088456)
(0.234375000000000,0.011565953285715)
(0.117187500000000,0.005825801152449)
(0.058593750000000,0.002672955303885)
(0.029296875000000,0.001240988325126)
(0.014648437500000,0.000650603226118)
(0.007324218750000,0.000339766051144)
(0.003662109375000,0.000175398659346)
(0.001831054687500,0.000085792698965)
(0.000915527343750,0.000033173978889)
};
\addplot[mark=none, green!70, thick] coordinates {
(1.875000000000000,0.019893377198335)
(0.937500000000000,0.008157201154251)
(0.468750000000000,0.000630026576381)
(0.234375000000000,0.000089925389707)
(0.117187500000000,0.000006448123106)
(0.058593750000000,0.000000939667949)
(0.029296875000000,0.000000113482046)
(0.014648437500000,0.000000021681987)
(0.007324218750000,0.000000004124020)
(0.003662109375000,0.000000001498943)
(0.001831054687500,0.000000000971768)
(0.000915527343750,0.000000000490682)
};
\addplot[mark=none, blue!70, thick] coordinates {
(1.875000000000000,0.035822571021985)
(0.937500000000000,0.015518120856720)
(0.468750000000000,0.000221914806123)
(0.234375000000000,0.000008960165742)
(0.117187500000000,0.000000312959706)
(0.058593750000000,0.000000021916342)
(0.029296875000000,0.000000007858108)
(0.014648437500000,0.000000008682709)
(0.007324218750000,0.000000009975433)
(0.003662109375000,0.000000008253390)
(0.001831054687500,0.000000005249710)
(0.000915527343750,0.000000002350192)
};
\addplot [dashed,domain=0.001:0.5, samples=100]{0.2*x};
\addplot [dashed,domain=0.01:0.7, samples=100]{0.02*x^3};
\addplot [dash dot dot,domain=0.075:0.75, samples=100]{0.0025*x^4.5};
\end{loglogaxis}
\end{tikzpicture}

%% file: plots/nvi-times.tex
\begin{tikzpicture}[scale = 0.75]
\begin{loglogaxis}[
    xlabel={number of time steps $N$},
    ylabel={online computation time},
    grid=major,
    legend entries={$p=1$,$p=2$,$p=3$, $O(N)$},
    legend pos=north west
]
\addplot[mark=none, red!70, thick] coordinates {
(32,     0.075303)
(64,     0.061800)
(128,    0.040439)
(256,    0.023111)
(512,    0.043082)
(1024,   0.022524)
(2048,   0.031244)
(4096,   0.047916)
(8192,   0.081919)
(16384,  0.154839)
(32768,  0.275848)
(65536,  0.543573)
(131072, 1.043702)
};
\addplot[mark=none, green!70, thick] coordinates {
(32,     0.554620)
(64,     0.113481)
(128,    0.054266)
(256,    0.039307)
(512,    0.076629)
(1024,   0.059572)
(2048,   0.096130)
(4096,   0.167859)
(8192,   0.318364)
(16384,  0.607379)
(32768,  1.189554)
(65536,  2.338551)
(131072, 4.628578)
};
\addplot[mark=none, blue!70, thick] coordinates {
(32,     0.093259)
(64,     0.081549)
(128,    0.067727)
(256,    0.052219)
(512,    0.083495)
(1024,   0.081842)
(2048,   0.134730)
(4096,   0.246592)
(8192,   0.450039)
(16384,  0.859248)
(32768,  1.696222)
(65536,  3.365783)
(131072, 6.721478)
};
\addplot [dashed,domain=1024:131072, samples=100]{0.00002*x};
%
\end{loglogaxis}
\end{tikzpicture}

%% file: plots/fd-convergence.tex
\begin{tikzpicture}[scale = 0.75]
\begin{loglogaxis}[
    xlabel={time step $h$},
    ylabel={$e_h$},
    grid=major,
    legend entries={$p=1$,$p=2$,$p=3$, $O(h^{2p-1})$},
    legend style={at={(1.225,0.4)}}
]
\addplot[mark=none, red!70, thick] coordinates {
(0.062500000000000, 0.021721800979836)
(0.031250000000000, 0.011122741056516)
(0.015625000000000, 0.005614388584481)
(0.007812500000000, 0.002820972605852)
(0.003906250000000, 0.001416461345902)
(0.001953125000000, 0.000550921904039)
(0.000976562500000, 0.000216172789916)
(0.000488281250000, 0.000078323566310)
(0.000244140625000, 0.000029395525566)
(0.000122070312500, 0.000015447264708)
(0.000061035156250, 0.000007910787756)
(0.000030517578125, 0.000003226179737)
};
\addplot[mark=none, green!70, thick] coordinates {
(0.062500000000000, 0.000158569890250)
(0.031250000000000, 0.000070291274707)
(0.015625000000000, 0.000009750202235)
(0.007812500000000, 0.000001277273725)
(0.003906250000000, 0.000000163025279)
(0.001953125000000, 0.000000049491687)
(0.000976562500000, 0.000000018743334)
(0.000488281250000, 0.000000008364506)
(0.000244140625000, 0.000000004112145)
(0.000122070312500, 0.000000002060748)
(0.000061035156250, 0.000000000965453)
(0.000030517578125, 0.000000000350518)
};
\addplot[mark=none, blue!70, thick] coordinates {
(0.062500000000000, 0.000128083850282)
(0.031250000000000, 0.000000933396858)
(0.015625000000000, 0.000000024614396)
(0.007812500000000, 0.000000001088977)
(0.003906250000000, 0.000000000823098)
(0.001953125000000, 0.000000000730466)
(0.000976562500000, 0.000000000632876)
(0.000488281250000, 0.000000000532636)
(0.000244140625000, 0.000000000455072)
(0.000122070312500, 0.000000000423848)
(0.000061035156250, 0.000000000362037)
(0.000030517578125, 0.000000000240015)
};
\addplot [dashed,domain=0.00003:0.06, samples=100]{1.0*x};
\addplot [dashed,domain=0.003:0.04, samples=100]{8*x^3};
\addplot [dashed,domain=0.01:0.08, samples=100]{5*x^5};
\end{loglogaxis}
\end{tikzpicture}

%% file: plots/fd-times.tex
\begin{tikzpicture}[scale = 0.75]
\begin{loglogaxis}[
    xlabel={number of time steps $N$},
    ylabel={online computation time},
    grid=major,
    legend entries={$p=1$,$p=2$,$p=3$, $O(N)$},
    legend pos=north west
]
\addplot[mark=none, red!70, thick] coordinates {
(32,       0.624277)
(64,       0.813853)
(128,      1.523191)
(256,      2.819560)
(512,      5.606501)
(1024,    11.156703)
(2048,    22.758282)
(4096,    53.521360)
(8192,   104.164685)
(16384,  233.195326)
(32768,  484.743312)
(65536,  984.398470)
(131072,1376.133322)
};
\addplot[mark=none, green!70, thick] coordinates {
(32,       1.068093)
(64,       1.622303)
(128,      3.380504)
(256,      6.344614)
(512,     12.857011)
(1024,    29.605384)
(2048,    49.096401)
(4096,   105.065974)
(8192,   204.271527)
(16384,  406.256171)
(32768,  845.943655)
(65536, 1709.507510)
(131072,3523.049906)
};
\addplot[mark=none, blue!70, thick] coordinates {
(32,       2.706178)
(64,       4.246933)
(128,      9.820791)
(256,     21.951591)
(512,     41.998223)
(1024,    76.988625)
(2048,   144.428000)
(4096,   220.065101)
(8192,   380.967206)
(16384,  752.828073)
(32768, 1514.840942)
(65536, 2679.623742)
(131072,5679.396925)
};
\addplot [dashed,domain=32:131072, samples=100]{0.005*x};
%
\end{loglogaxis}
\end{tikzpicture}